\documentclass[onesided]{amsart}
\usepackage{amssymb,latexsym,amsmath,amsthm, mathtools,enumerate}
\usepackage[margin=1in]{geometry}
\usepackage{fancyhdr}
\usepackage{color}
\usepackage{mathrsfs}
\usepackage{hyperref}
 
\DeclarePairedDelimiter{\abs}{\lvert}{\rvert}
\DeclarePairedDelimiter{\norm}{\lVert}{\rVert}

\DeclarePairedDelimiter{\parens}{(}{)}
\DeclarePairedDelimiter{\floor}{\lfloor}{\rfloor}
\newcommand{\ov}{\overline}
\newcommand{\up}[1]{^{(#1)}}
\newcommand{\ord}{\operatorname{ord}}
\newcommand{\ch}{\operatorname{char}}
\newcommand{\codim}{\operatorname{codim}}
\newtheorem{letterthm}{Theorem}

\newtheorem{lettercor}[letterthm]{Corollary}

\newtheorem{thm}{Theorem}[section]
\newtheorem*{thm*}{Theorem}
\newtheorem{prop}[thm]{Proposition}
\newtheorem{lemma}[thm]{Lemma}
\newtheorem{cor}[thm]{Corollary}

\newtheorem{condition}{Condition}[section]
\newtheorem*{condition*}{Condition}

\theoremstyle{remark}
\newtheorem{remark}{Remark}[subsection]

\newcommand{\ep}{\varepsilon}

\newcommand{\con}{\equiv}

\newcommand{\ndiv}{\nmid}
\newcommand{\modd}[1]{\; ( \mathrm{mod} \; #1)}
\newcommand{\bstack}[2]{#1 \atop #2}

\newcommand{\maps}{\rightarrow}
\newcommand{\intersect}{\cap}

\newcommand{\union}{\cup}

\newcommand{\Gal}{{\rm Gal }}

\newcommand{\al}{\alpha}
\newcommand{\be}{\beta}

\newcommand{\del}{\delta}
\newcommand{\Del}{\Delta}

\newcommand{\om}{\omega}

\newcommand{\sig}{\sigma}

\newcommand{\Acal}{\mathcal{A}}
\newcommand{\Pcal}{\mathcal{P}}

\newcommand{\Dcal}{\mathcal{D}}

\newcommand{\Mcal}{\mathcal{M}}

\newcommand{\Scal}{\mathcal{S}}

\newcommand{\Lbf}{\mathbf{L}}
\newcommand{\Hbf}{\mathbf{H}}
\newcommand{\Mbf}{\mathbf{M}}

\newcommand{\Nbf}{\mathbf{N}}
\newcommand{\Kbf}{\mathbf{K}}
\newcommand{\Vbf}{\mathbf{V}}
\newcommand{\Xbf}{\mathbf{X}}
\newcommand{\delbf}{\boldsymbol\delta}
\newcommand{\onebf}{\boldsymbol1}
\newcommand{\zerobf}{\boldsymbol0}

\newcommand{\abf}{{\bf a}}

\newcommand{\kbf}{\mathbf{k}}

\newcommand{\mbf}{{\bf m}}

\newcommand{\tbf}{{\bf t}}
\newcommand{\ubf}{{\bf u}}
\newcommand{\vbf}{{\bf v}}

\newcommand{\xbf}{{\bf x}}

\newcommand{\Qbar}{\overline{\Q}}

\newcommand{\Abb}{\mathbb{A}}
\newcommand{\C}{\mathbb{C}}
\newcommand{\F}{\mathbb{F}}

\newcommand{\N}{\mathbb{N}}

\newcommand{\Q}{\mathbb{Q}}
\newcommand{\R}{\mathbb{R}}

\newcommand{\Z}{\mathbb{Z}}

\newcommand{\Tr}{\text{Tr}}

\newcommand{\beq}{\begin{equation}}
\newcommand{\eeq}{\end{equation}}
\numberwithin{equation}{section}

\begin{document}

\title{Burgess bounds for short character sums evaluated at forms }

\author[Pierce]{Lillian B. Pierce}
\address{Department of Mathematics, Duke University, 120 Science Drive, Durham NC 27708 USA}
\email{pierce@math.duke.edu}

\author[Xu]{Junyan Xu}
\address{Department of Mathematics, Indiana University, 831 East 3rd Street, Bloomington, IN 47405}
\email{xu56@indiana.edu}

\maketitle  

\begin{abstract}
In this work we establish a Burgess bound for  short multiplicative character sums in arbitrary dimensions, in which the character is evaluated at  a homogeneous form that belongs to a very general class of ``admissible'' forms.  This $n$-dimensional Burgess bound  is nontrivial for sums over boxes of sidelength at least $q^{\be}$, with $\be  > 1/2 - 1/(2(n+1))$. This is the first  Burgess bound that applies in all dimensions to generic forms of arbitrary degree. Our approach capitalizes on a recent stratification result for complete multiplicative character sums evaluated at rational functions, due to the second author.
\end{abstract}

\section{Introduction}\label{sec_forms}

The celebrated Burgess bound \cite{Bur57} proves that for $\chi$ a non-principal multiplicative Dirichlet character   modulo a prime $q$, the   character sum
\[ S(N,H) = \sum_{N < n \leq N+H} \chi(n)\]
is bounded for every integer $r \geq 1$ by
\beq\label{Burgess_1}
  |S( N,H)| \ll_r  H^{1 - 1/r} q^{ \frac{r + 1}{4r^2}} \log q.
  \eeq
 From this it can be deduced that $S(N,H)$ 
admits a nontrivial bound $o(H)$ for $H$ as small as
$H=q^{1/4+\kappa},
$
 for any $\kappa>0$. Bounds for $S(N,H)$ have many applications, and as we survey in \S \ref{sec_overview}, Burgess's influential work set records that remain the best known today.

 This paper proves the first $n$-dimensional Burgess bound for short multiplicative character sums  evaluated at generic homogeneous polynomial arguments of arbitrarily large degree. Let $\chi$ be a non-principal character modulo a prime $q$.
Let $n \geq 1$ be a fixed dimension, and $F \in \Z[x_1,\ldots,x_n]$  a form of degree $D$. For any $\Nbf =(N_1,\ldots, N_n)$, $\Hbf = (H_1,\ldots,H_n) \in \R^n$, define
\beq\label{Burgess_n}
 S(F;\Nbf,\Hbf) = \sum_{\bstack{\xbf \in \Z^n}{x_i \in (N_i,N_i+H_i]}} \chi(F(\xbf)).
 \eeq
 Given $\Hbf$, we will define $\|\Hbf\| = H_1\cdots H_n$, so that in particular, $\|\Hbf\|$ is a trivial bound for $|S(F;\Nbf,\Hbf)|$. 
  Previous to the work of this paper, when the lengths $H_i$ are short, that is $\ll q^{1/2+\ep},$ nontrivial bounds for $S(F;\Nbf,\Hbf)$ of the form $o(\|\Hbf\|)$   were known only in special cases, such as when $F$ is a product of $n$ linear forms that are linearly independent over $\F_q$, or when $n=2$ and $F$ is a binary quadratic form (see \S \ref{sec_lit} for details). In this paper, we prove nontrivial bounds for $S(F;\Nbf,\Hbf)$ in any dimension $n$ for a very general class of ``admissible'' forms $F$, as long as 
   \[ \| \Hbf \| H_{\min} \gg q^{n/2 + \kappa},\]
 for some $\kappa>0$,  where $H_{\min} = \min_i H_i$. 
In particular, this is satisfied when $\Hbf = (H,\ldots, H)$ with $H=q^{\be_n + \kappa}$ for any $\kappa>0$, where
    \beq\label{beta_dfn_00}
 \be_n = \frac{1}{2} - \frac{1}{2(n+1)}.
 \eeq

\subsection{Statement of the main theorem}

We now provide a formal statement of the condition that a form $F$ must satisfy in order to be ``admissible'' for our main result.
 We only need to rule out those forms $F$ for which a nontrivial bound for $S(F;\Nbf,\Hbf)$ would naively fail, such as when $F$ is a perfect $\Del$-th power and $\chi$ is order $\Del$, or when $F$ can be made to depend on fewer than $n$ variables.
 
  \begin{condition}[$(\Del,q)$-admissible]\label{cond} Let $q$ be a fixed prime and $\Del \geq 1$ a fixed integer. We will say that a form $F \in \Z[x_1,\ldots, x_n]$ is $(\Del,q)$-admissible if the following holds. Let $f$ denote the the reduction of $F$ modulo $q$, so that we may consider $f \in \F_q[x_1,\ldots, x_n]$. Factorize $f = g^\Del h$ where $g,h \in \F_q[x_1,\ldots, x_n]$ and $h$ is $\Del$-th power-free over $\F_q$. Then $h$ has the property that it cannot be made independent of (at least) one variable after a linear transformation, i.e.  there exists no linear change of variables $A \in \mathrm{GL}_n(\F_q)$  such that $h(\xbf A) \in \F_q[x_2,\ldots, x_n]$. 
      \end{condition}

See \S \ref{sec_prelim} for further details on this condition, and a precise definition of being $\Del$-th power-free. 
If $\Del\geq 2$ is a fixed integer, any form $F \in \Z[x_1,\ldots,x_n]$ that has the property that $F=G^\Del H$ with $G,H \in \Z[x_1,\ldots,x_n]$ where $H$ is $\Del$-th power-free and nondegenerate with respect to changes of variables in $\mathrm{GL}_n(\Z)$, is $(\Del,q)$-admissible for all but finitely many primes $q$  (see Lemma \ref{lemma_F_ae_q}).  For any $D$, the form $x_1^D + \cdots + x_n^D$ is an example of such a form.
Moreover, such forms are generic among all forms in $\Z[x_1,\ldots,x_n]$ of degree at most $D$, since those that violate the conditions depend on fewer parameters.

Our main result is as follows:
\begin{thm}\label{thm_main_mult}
 Let $\chi$ be a non-principal multiplicative Dirichlet character of order $\Delta$ modulo a prime $q$. Let $n \geq 2$ be fixed. For each $r \geq 1$, define 
 \[ \Theta = \Theta_{n,r} =   \floor*{\frac{r-1}{n-1}} .\]
 Let $\Hbf=(H_1,\dots,H_n)\in\R_{\geq 1}^n$ have maximum element $H_{\max}$ and minimum element $H_{\min}$. Then as long as $H_{\max}H_{\min}<q^{1+1/(2\Theta)}$, we have for all degree $D$ forms $F\in\Z[x_1,\dots,x_n]$ that are $(\Del,q)$-admissible, that uniformly in $\Nbf = (N_1,\ldots, N_n)$, for every integer $r \geq 1$,
\beq\label{thm_bound_H}
\abs{S(F;{\bf N,H})}\ll \norm{\Hbf}^{1-\frac{1}{2r}}H_{\min}^{-\frac{1}{2r}}q^{\frac{n\Theta+1}{4r\Theta}}(\log q)^{n+1}, 
\eeq
 in which the implied constant depends only on $D, \Del, n, r$ and is otherwise independent of $F$.
\end{thm}

If one carries through our method of  proof in the case $n=1$, we may define $\Theta = \Theta_{1,r} =r$ for all $r \geq 1$ and recover the Burgess bound (\ref{Burgess_1}) (up to the power of the logarithm). This result may also be extended to apply to a rational function $F = f_1/f_2$ if it is appropriately regarded as $f_1 f_2^{\Del-1}$. Analogous to the proof of the P\'olya-Vinogradov inequality, Fourier-based methods  can prove a nontrivial bound for $S(F;\Nbf,\Hbf)$ for suitable forms $F$  when $H_i \gg q^{1/2+\ep}$ (see \S \ref{sec_Fourier}); thus the upper bound restriction on $H_i$ in the hypothesis of the theorem is compatible with our interest in the range $H_i \ll q^{1/2+\ep}$.

For purposes of comparison, we state a direct corollary of Theorem \ref{thm_main_mult} in the case that all coordinates of $\Hbf$ are of equal size.

\begin{cor}\label{thm_main_mult_cor}
Assume the hypotheses of Theorem \ref{thm_main_mult}. Then for $\Hbf = (H,H,\ldots,H)$ with $H<q^{1/2+1/(4\Theta)}$, we have   for all degree $D$ forms $F\in\Z[x_1,\dots,x_n]$ that are $(\Del,q)$-admissible, that uniformly in $\Nbf = (N_1,\ldots, N_n)$, for every integer $r \geq 1$,
\beq\label{cor_bound_H}
\abs{S(F;{\bf N,H})}\ll H^{n- \frac{n+1}{2r}}q^{\frac{n\Theta+1}{4r\Theta}}(\log q)^{n+1}, 
\eeq
 in which the implied constant depends only on $D, \Del, n, r$  and is otherwise independent of $F$.
\end{cor}

\subsection{The strength of Theorem  \ref{thm_main_mult}: quantifications}\label{sec_strength}

In general, for a Burgess-style result such as (\ref{Burgess_1}),which holds for a range of integers $r$, to assess its strength for $H$ near the lower-bound threshold allowed for $H$ in terms of $q$, we must compute which value of $r$ produces maximum savings. For example, in Burgess's original result, 
if $H=q^{1/4 +\kappa}$ the bound (\ref{Burgess_1}) with the parameter $r$ yields the upper bound $|S(N,H)| \ll Hq^{-\del}$ where  $\del = (4\kappa r -1)/(4r^2)$. Computing the maximum of $\del$ with respect to $r$, we see that by choosing $r$ to be the nearest integer to $1/(2\kappa)$ we may obtain the best value  $\del \approx \kappa^2$. 

We  perform an analogous optimization of our result, summarized in two corollaries.
 \begin{cor}\label{cor_threshold_same}
 For each $n \geq 2$ and $r \geq 1$, Theorem \ref{thm_main_mult} with the parameters $n,r$ provides a nontrivial upper bound 
$ |S(F;\Nbf,\Hbf)|= o_{n,r,\Del,D}(\|\Hbf\|)$
 for all $\Hbf = (H,\ldots, H)$ with $H=q^{\be}$ with $\be$ in the  range
 \beq\label{cor_range_beta_n}
  \frac{1}{2} - \frac{\Theta-1}{2\Theta(n+1)} < \be \leq \frac{1}{2} + \frac{1}{4\Theta}, 
  \eeq
 in which $\Theta = \Theta_{n, r} = \lfloor (r-1)/(n-1) \rfloor.$  This range includes a non-empty interval of $\be<1/2$ as soon as $r \geq 2n-1$, so that $\Theta=\Theta_{n, r}>1$.
 In particular,  this range always requires $\be > \be_n$ with 
 \beq\label{beta_dfn_0}
 \be_n = \frac{1}{2} - \frac{1}{2(n+1)}.
 \eeq
  For $H=q^{\be_n + \kappa}$ for small $\kappa$, we obtain a nontrivial bound $\|\Hbf\| q^{-\del_n}$ with savings approximately of size
 \[\del_n \approx  \frac{(n+1)^2}{4(n-1)}\kappa^2 .\]
 \end{cor}	
 The threshold $\be_n$ defined in (\ref{beta_dfn_0}) has appeared in  $n$-dimensional Burgess bounds that were previously proved in very special cases, such as \cite{Bur67,Bur68,DavLew63} (see  \S \ref{sec_lit} for an overview, including the stronger results \cite{Cha08,Cha09,BouCha10,HB16}).

In full generality, our main result Theorem \ref{thm_main_mult} is in fact stronger than Corollary \ref{cor_threshold_same}, as it can allow one or more of the lengths $H_i$ to be smaller than $q^{\be_n}$, as long as other $H_i$ are commensurably larger. We cannot let $H_i$ vary in a completely uncontrolled fashion, since our savings comes from the smallest parameter $H_{\min}$; thus we assume that ${\bf H}$  is proportionate, in the sense that $H_{\min}\ge\norm{{\bf H}}^{c_0/n}$ for some constant $0<c_0\le1$. The relation (\ref{be_relation}) below shows that if we take $c_0$ smaller so that $H_{\min}$ becomes a decreasing proportion of $\|\Hbf\|$, the restriction on the geometric mean $\|\Hbf\|^{1/n}$ so that our Burgess bound is nontrivial, is forced into an ever shorter range near $\|\Hbf\|^{1/n}=q^{1/2}$.

  \begin{cor}\label{cor_threshold_diff}
 For each $n \geq 2$ and $r\geq 1$, Theorem \ref{thm_main_mult} with the parameters $n,r$ provides a nontrivial upper bound 
$ |S(F;\Nbf,\Hbf)|= o_{n,r,\Del,D}(\|\Hbf\|)$
 for all $\Hbf = (H_1,\ldots, H_n)$ with $H_{\min} H_{\max} < q^{1/2 + 1/2\Theta}$ as long as $H_{\min}  \geq \| \Hbf \|^{c_0/n}$ for some $0< c_0 \leq 1$ and $\|\Hbf\|^{1/n} = q^\be$ with $\be$ in the   range
 \beq\label{cor_range_beta_n'}
   \frac{1}{2} - \frac{c_0\Theta-1}{2\Theta(n+c_0)} < \be \leq \frac{1}{2} + \frac{1}{4\Theta}, 
   \eeq
 in which $\Theta = \Theta_{n, r} = \lfloor (r-1)/(n-1) \rfloor.$  This range includes a non-empty interval of $\be<1/2$ as soon as $r \geq (1/c_0 + 1)(n-1)+1$, so that $c_0 \Theta >1$. In particular this range always requires $\|\Hbf\|^{1/n} = q^\be$ with 
 \beq\label{be_relation}
 \be > \be_{n,c_0} = \frac{1}{2} - \frac{c_0}{2(n+c_0)} \geq \frac{1}{2} - \frac{1}{2(n+1)}.
 \eeq
 Alternatively, we can state that Theorem \ref{thm_main_mult} obtains a nontrivial upper bound if 
 \[ \| \Hbf \| H_{\min} \gg q^{n/2 + \kappa}\]
  for some small $\kappa$.  As $\kappa \maps 0$ we obtain a nontrivial bound $\|\Hbf \|q^{-\del_n}$ with 
 \[ \del_n \approx  \frac{(n+c_0)^2}{4(n-1)}\kappa^2 .\]
 \end{cor}

\subsection{Overview of   previous literature}\label{sec_overview}
To situate our results, we recall previous literature on Burgess bounds and in particular for the sums $S(F;\Nbf,\Hbf)$.
\subsubsection{The classical Burgess bound}\label{sec_apps}
For any integer $q \geq 1$, the P\'olya-Vinogradov inequality states that 
$ |S(N,H)| \ll q^{1/2} \log q$ (see \cite{Pol18,Vin18} or for a modern treatment \cite[\S12.4]{IwaKow04}).
 This provides a nontrivial upper bound for $|S(N,H)|$ as long as $H \gg q^{1/2+\ep}$ for some $\ep>0$. When $H$ is shorter than this range, the sum is considered to be ``short,'' and obtaining an $o(H)$ bound is much more difficult. Conjecturally, under the Generalized Riemann Hypothesis, a bound as strong as
$ |S(0,H)| \ll_\ep H^{1/2} q^\ep$
  should hold for all $\ep>0$, thus leading to a nontrivial upper bound in any range $H \gg_\ep q^{3\ep}$ (see e.g. \cite[Eqn. (12.54)]{IwaKow04}, or \cite[Eqn. (9.6)]{FIM13}; see also the more general Conjecture $C_n$ in \cite[\S 9]{FIM13}).

Burgess's works \cite{Bur57,Bur62B,Bur63A,Bur86}
show that for $\chi$ a primitive character to a prime modulus $q$, for all integers $r \geq 1$, the inequality (\ref{Burgess_1}) holds, 
with an implied constant uniform in $N$, yielding a nontrivial bound for $H\gg q^{1/4+\kappa}$, $\kappa>0$. More generally, with $\log q$ replaced by $q^\ep$ for arbitrarily small $\ep$, Burgess proved that this bound also holds for cube-free moduli $q$ for all $r \geq 1$, and for any integer $q$, for $r \leq 3$. 
The Burgess bound remains essentially unimproved since its inception, despite significant interest, due to its applications.

 As a  consequence of (\ref{Burgess_1}), Burgess \cite{Bur63A} proved a landmark subconvexity bound 
  \[ |L(1/2 +it,\chi)| \ll_{t,\ep} q^{1/4-1/16 + \ep} \]
  for all $\ep>0$, with $\chi$ a non-principal Dirichlet character modulo $q$ as above; there is a corresponding hybrid subconvexity bound $|L(1/2+it,\chi)| \ll_\ep (|t|q)^{1/4 - 1/16 + \ep}$ of Heath-Brown \cite{HB80}.  This remains the best bound known to hold for all Dirichlet $L$-functions. (Special cases of the modulus $q$ in which  a better subconvexity bound is known include:  smooth moduli \cite{GraRin90,Gol10,Cha14} and most recently \cite{Irv16};  prime-power moduli  \cite{Pos56,BLT64,Gal72} and most recently \cite{Mil16},   or powerful moduli \cite{Iwa74}. Most recently, Petrow and Young have proved a better Weyl-strength subconvex estimate of size $q^{1/6}$ for all cube-free moduli $q$ \cite{PetYou18x}.) There is great interest in establishing bounds of at least this strength in more generality; this establishes the notion of a ``Burgess exponent'' for analogous bounds in higher rank contexts. For example, in the $\mathrm{GL}(2)$ setting, the ``Burgess bound'' for an $L$-function of a Hecke cusp form $g$ twisted by a primitive Dirichlet character $\chi$ modulo $q$ is
  $|L(1/2+it,g \otimes \chi) | \ll_{g,\ep} q^{1/2 - 1/8 + \ep}$ for all $\ep>0$ (as has been obtained in \cite{Byk96, BloHar08}). Reaching the Burgess exponent in new settings, or even re-proving such  Burgess bounds is currently an important proving ground for new methods (e.g. Munshi \cite{Mun17x} via a $\mathrm{GL}(2)$ delta method and subsequently \cite{AHLS18x} via a trivial delta method).

In another direction, the Burgess bound establishes an upper bound for the least quadratic non-residue $n(p)$ modulo a prime $p$. Vinogradov conjectured that  $n(p) \ll_\ep p^\ep$ for every $\ep>0$; Burgess's bound (\ref{Burgess_1}) proves $n(p) \ll_\ep p^{(4\sqrt{e})^{-1} + \ep}$ for any $\ep>0$, which held the record from \cite{Bur57} until the quantification in \cite{BanGuo17}.
In this vein, there are continued efforts  toward the goal of improving the inequality (\ref{Burgess_1}) directly, such as reducing the power of the logarithm (see \cite[Eqn. (12.58) and Remark p. 329]{IwaKow04} and \cite{KSY17x}),
deducing improvements in certain special cases from conjectural improvements on the P\'olya-Vinogradov inequality \cite{FroGol17x}, and making connections to the Elliott-Halberstam conjecture and ``Type II sums'' in sieve methods  \cite[Conjecture 1.5, Thm. 1.6, Remark 1.7]{Tao15}.

\subsubsection{Previous literature on special cases of $S(F; \Nbf,\Hbf)$}\label{sec_lit}
In the $n$-dimensional setting of the sums $S(F;\Nbf,\Hbf)$,  previous literature  mainly focused on two special cases. Burgess \cite{Bur67,Bur68} considered the case in which 
\beq\label{multilinear}
F(\xbf) = \prod_{i=1}^n L_i(\xbf)
\eeq
 is a product of $n$ linear forms $L_i \in \Z[x_1,\ldots, x_n]$ that are linearly independent over $\F_q$, for $q$ prime.   In this multilinear setting, he proved a nontrivial bound $|S(F;\Nbf,\Hbf)|  = O( \|\Hbf\|q^{-\del})$ for $\Hbf = (H,\cdots,H)$ and a certain $\del = \del(\kappa)>0$ as long as $H =  q^{\be_n + \kappa}$ for some $\kappa>0$, with $\be_n$ as defined in (\ref{beta_dfn_0}).
 Bourgain and Chang \cite{BouCha10} incorporated ideas from additive combinatorics to improve this significantly, proving a nontrivial bound $|S(F;\Nbf,\Hbf)|  = O( \|\Hbf\|q^{-\del})$ for $\Hbf = (H,\cdots,H)$ and a certain $\del = \del(\kappa)>0$ as long as $H =  q^{1/4 + \kappa}$ for some $\kappa>0$, thus obtaining an $n$-dimensional result as strong as the original Burgess threshold in each dimension.
 
The second case in which significant results are known is in dimension $n=2$ when $F$ is a binary quadratic form. 
 In this special case, the work of Burgess above, for bilinear sums in $n=2$, combined with results of  Davenport and Lewis \cite{DavLew63} on analogues of the Burgess bound over $\F_{q^2}$, initially provided a nontrivial upper bound for $S(F,\Nbf,\Hbf)$  for $q$ prime and $H_i> q^{1/3+\kappa}$ (that is $H_i> q^{\be_2+\kappa}$ with $\be_2$ as in (\ref{beta_dfn_0})), where $F$ is any binary quadratic form that is not a perfect square over $\F_q$; that is for any $F(x_1,x_2) = x_1^2+ ax_1x_2 + bx_2^2$ with $a^2 \not\con 4b \modd{q}$. 
Chang \cite[Thm. 11]{Cha09} introduced ideas of additive combinatorics to this setting, and improved this to a nontrivial upper bound for $S(F,\Nbf,\Hbf)$ for $H_1,H_2> q^{1/4+\kappa}$, i.e. a 2-dimensional result as strong as the original Burgess threshold in each dimension. Most recently, Heath-Brown proved that this latter result continues to hold for any odd square-free modulus $q$ such that $(q,\det(F))=1$ \cite[Thm. 3]{HB16}.

 It remains an interesting open question to bound  $S(F;\Nbf,\Hbf)$ in the fully general case of  $(\Del,q)$-admissible forms $F$, with all $H_i$ as short as $q^{1/4+\kappa}$ for $\kappa>0$.

 \subsubsection{Further related literature}\label{sec_related_lit}
We briefly mention certain other results that are related to multivariate sums similar to $S(F;\Nbf,\Hbf)$, although not of exactly the type we consider in this paper. Davenport and Lewis \cite{DavLew63} considered the case of $\chi$ a nonprincipal character of $\F_{q^n}^*$ for $q$ prime, and a linear form $F(x_1,\ldots, x_n) = \om_1 x_1 + \cdots + \om_n x_n$ for $\om_1,\ldots, \om_n$ a fixed basis of $\F_{q^n}$. They  proved that 
$S(F,\Nbf,\Hbf)  = O (\|\Hbf\| q^{-\del})$ for some $\del = \del(\kappa)$ where $H_i > q^{\be_n + \kappa}$ for some $\kappa>0$. 
As Burgess remarked in \cite{Bur68}, for $n=2$ this provides a corresponding upper bound for $S(F;\Nbf,\Hbf)$ in the case that $F$ is an irreducible binary quadratic form $x_1^2 +ax_1x_2+bx_2^2$ over $\F_q$, in which case $\chi(F(x_1,x_2))$ is a character $\modd{q}$ of $x_1 + \om x_2 \in \Q(\om)$ where $\om =(1/2) a + (1/2) \sqrt{a^2-4b}$.
 In the $n$-dimensional setting, if one assumes that $\om_1,\ldots, \om_n$ is a certain special type of basis (such as a power basis), stronger results have also been obtained in this setting; Burgess \cite{Bur67b} and Karatsuba \cite{Kar70} proved nontrivial upper bounds in the stronger range $H>q^{1/4+\kappa}$.  
Without such special assumptions, Chang  improved on Davenport and Lewis (for $n \geq 5$) by proving a nontrivial bound as soon as $\|\Hbf\| > q^{2n/5+\kappa}$ for some $\kappa>0$ \cite{Cha08}, and furthermore in dimension $n=2$ Chang proved that $H_i \gg q^{1/4+\kappa}$ suffices for any $\kappa>0$ \cite[Thm. 5]{Cha09}.  See also \cite{Cha08} for certain results of Burgess-type for multiplicative character sums over sumsets. 

Finally, we mention work on mixed character sums in multivariate settings, of the form $S(F,g;\Nbf,\Hbf) = \sum_{\xbf \in (\Nbf,\Nbf+\Hbf]} e(g(\xbf))\chi(F(\xbf))$, with a polynomial $g \in \R[x_1,\ldots, x_n]$. In \cite{Pie16} the first author proved nontrivial upper bounds for such multivariate sums in the regime $H_i \gg q^{1/4 +\kappa}$, in the special case $F(\xbf) = x_1\cdots x_n$ (generalizing \cite{HBP15} in dimension $n=1$). This was later generalized by Kerr \cite{Ker14x} to the case of $F(\xbf)$ multilinear as in (\ref{multilinear}). A second 
paper in this series will prove Burgess bounds for $S(F,g;\Nbf,\Hbf)$ for any $(\Del,q)$-admissible form $F$.

\subsection{Outline of the paper}
We present in \S \ref{sec_heuristic} a heuristic overview of the proof of Theorem \ref{thm_main_mult}, which illustrates how the stratification result of the second author \cite{Xu18} plays a key role. 
In \S \ref{sec_prelim} we gather together the lemmas we need to motivate and utilize the condition of $(\Del,q)$-admissibility. In \S \ref{sec_Xu} we give a convenient restatement and strengthening of the stratification results \cite{Xu18} of the second author; we expect this version will be of independent interest in other applications. In \S \ref{sec_Burgess_argument} we begin the $n$-dimensional Burgess argument, reaching the key new novel steps in \S \ref{sec_MR}, which carries out the stratification and a Menchov-Rademacher argument involving permuting variables. In \S \ref{sec_Burgess_conclude} we complete the Burgess argument and choose parameters optimally; subsequently we verify the corollaries. We provide a brief appendix \S \ref{sec_compare} including a conditional argument that assumes a stronger stratification result, which shows that the threshold $\be_n$ is stable under such an improvement.

\subsection{Notation}
We will use the notation $\xbf \in (\Nbf, \Nbf + \Hbf]$ to denote the range of the sum over a box $(\Nbf, \Nbf+\Hbf] = \prod_i (N_i,N_i+H_i]$, and will let $\| \Hbf \| = \prod_i H_i$ for any tuple $\Hbf$. We will write $\abf q$ to mean $(a_1q,\ldots, a_nq)$ and $\abf/q$ to mean $(a_1/q,\ldots, a_n/q)$. We will also use notations such as $\xbf \leq \abf$ to denote $x_i \leq a_i$ for $i=1,\ldots, n$ and $\abf \modd{p}$ to mean we regard each $a_i \modd{p}$. We define $H_{\max}:=\max_i H_i$ and $H_{\min}:=\min_i H_i$. 
We will use the Vinogradov notation $A \ll B$ to denote that there exists a constant $C$ such that $|A | \ll CB$, and $A\ll_\kappa B$ to denote that $C$ may depend on the parameter $\kappa$. In the following work, all implied constants may depend on $n, r, D, \Delta:=\ord \chi$ and possibly an arbitrarily small $\ep>0$ without further specification,  but will never depend on $\Nbf, \Hbf, q$.

\section{Method of proof: an overview}\label{sec_heuristic}
In this section we recall the main points of the Burgess method in dimension 1,   outline the difficulties that arise in $n \geq 2$ dimensions, and then sketch how we overcome these difficulties to obtain our main theorem.

At its heart, the Burgess method in the classical $1$-dimensional setting builds  from a ``short'' character sum $ S(N;H)$ of length $H \ll q^{1/2+\ep}$, a ``long'' character sum over a complete set of residues modulo $q$. Doing so by a Fourier expansion only works efficiently if the sum is not too short, that is, if the character sum is of length $H$ with at least $H\gg q^{1/2+\ep}$ in length (see \S \ref{sec_Fourier}).  Burgess's method instead dissects and translates the sum into many ``short-short'' sums of length $H/p$ for some prime $p$ of size roughly $P$ (with $P$ to be chosen optimally in terms of $H$ and $q$), and averaging this process over sufficiently many choices of $p$, these short-short sums would become distributed across a long interval of length $q$. If this process is done with enough redundancy, the starting points of these short-short sums nearly cover a full set of residues modulo $q$.  Simultaneously, Burgess then considers not just an average of these short-short sums, but a $2r$-th moment, leading to the study of
\beq\label{S_heuristic1'}
\sum_{m} \max_{k \leq 2 H /P} |S(m,k)|^{2r}.
\eeq
At this point, positivity allows one to sum over all $1 \leq m \leq q$ so that the sum over $m$ is a complete set of residues.  The Menchov-Rademacher technique allows one to deduce a bound for this maximal moment (\ref{S_heuristic1'}) from a bound for the non-maximal moment 
\beq\label{S_heuristic2'}
\sum_{m \modd{q}}  |S(m,k)|^{2r} 
 \eeq
in which we think of $m$ as varying over the starting points of the short-short sums and $k \leq 2H/P$ as being the new short-short length.
We may write (\ref{S_heuristic2'}) equivalently as 
\beq\label{S_heuristic3'}
 \sum_{x_1,\ldots, x_{2r} \in (0,k]} \sum_{m \modd{q}} \chi(F_{\xbf}(m)),
 \eeq
 in which $F_{\xbf}(m) = (m+x_1)(m+x_2)^{\Del-1} \cdots (m+x_{2r})^{\Del-1}$. If $F_{\xbf}(m)$ is not a perfect $\Del$-th power modulo $q$, then the Weil bound $O(q^{1/2})$  applies to the sum over $m$, and we say the tuple $\xbf$ is ``good'' (which is the generic case); otherwise $\xbf$ is ``bad'' (which is a sparse case) and we apply the trivial  bound $O(q)$ to the sum over $m$. Balancing the contributions of these two cases leads to the optimal choice of $P$ and the Burgess bound (\ref{Burgess_1}).

Generalizing this argument to the $n$-dimensional case, we will prove in (\ref{SFNH_D4}) that for any $r \geq 1$,
 \beq\label{SFNH_D4_intro}
  |S(F;\Nbf,\Hbf)| \ll  (\log P)P^{n-1/2r}  \| \Hbf \|^{-1/2r} \left(\sum_{\bstack{\mbf}{ |m_i| \leq 2q}} \max_{\kbf \leq 2\Hbf/P} |S(F;\mbf,\kbf)|^{2r}\right)^{1/(2r)} .
 \eeq
 This will ultimately reduce the problem of bounding $|S(F;\Nbf,\Hbf)|$ to bounding 
\beq\label{S_heuristic2_intro}
 \sum_{\mbf \modd{q}} |S(F;\mbf,\kbf)|^{2r}
   \leq \sum_{\bstack{\xbf^{(1)},\ldots,\xbf^{(2r)} }
   {\xbf\up i\in(\zerobf,\kbf]}} \left| \sum_{\mbf \modd{q}} \chi (F_{\{\xbf\}}(\mbf)) \right|,
 \eeq
where we define for each collection $\{\xbf\}  = \{\xbf^{(1)},\ldots,\xbf^{(2r)} \}$  with $\xbf^{(i)} \in \Z^n$ the polynomial
\beq\label{dfn_Gpoly}
 F_{\{\xbf\}} (\Xbf) = F_{ \{\xbf\}} (X_1,\ldots, X_n) = \prod_{i=1}^{2r} (F(\Xbf + \xbf^{(i)}))^{\del(i)}, 
 \eeq
where $\del(i)=+1$ if $i$ is odd and $\Delta -1$ if $i$ is even, where $\Del$ is the order of $\chi$ modulo $q$. One would hope that if the  $\xbf^{(i)}$ are appropriately independent (that is, the ``good'' case),  a  generalization of the Weil bound would yield square-root cancellation,  that is an $O(q^{n/2})$ bound for 
\beq\label{char_sum_intro}
 \left|\sum_{\mbf \modd{q}}
\chi(F_{\{\xbf\}}(\mbf))\right|.
\eeq
But achieving such a bound has been a critical barrier to generalizing the Burgess method to this $n$-dimensional setting. One difficulty is that the leading form of $F_{\{\xbf\}}(\Xbf)$ (the homogeneous part of highest degree) defines a highly singular projective variety, whereas previous literature on  Weil bounds required this either to be a nonsingular projective variety \cite{Kat02}, or could only allow certain singular varieties that are not general enough for our application; see e.g. \cite{Roj05} and \cite{Roj06}.

Moreover, in dimensions $n \geq 2$, as well as the two extremal cases in which the sum (\ref{char_sum_intro}) is $O(q^{n/2})$ or $O(q^n)$, there may be intermediate cases $O(q^{(n+j-1)/2})$ for $j=1,\ldots, n+1$. Indeed, suppose we partition the $\{ \xbf^{(1)},\ldots,\xbf^{(2r)} \} \in (\zerobf,\kbf]^{2r}$  into the following types: those  belonging to a good set denoted by $\mathrm{Good}(\kbf)$, are such that  
\[ \left| \sum_{\mbf \modd{q}} \chi (F_{\{\xbf\}}(\mbf))\right| \leq C q^{n/2}\]
for a certain constant $C$,
and those belonging to the $j$-th bad set, denoted by $\mathrm{Bad}_j(\kbf)$, are such that 
\[\left| \sum_{\mbf \modd{q}} \chi (F_{\{\xbf\}}(\mbf))\right| > C q^{(n+j-1)/2} \qquad \text{but} \qquad \left| \sum_{\mbf \modd{q}} \chi (F_{\{\xbf\}}(\mbf))\right| \leq C q^{(n+j)/2}.\]
Then according to this dissection,
\beq\label{good_bad_model}
  \sum_{\mbf \modd{q}} |S(F;\mbf,\kbf)|^{2r} \leq C|\mathrm{Good}(\kbf)|q^{n/2} + C \sum_{j=1}^n|\mathrm{Bad}_j(\kbf)| q^{(n+j)/2}.
  \eeq
 
  Certainly $|\mathrm{Good}(\kbf)|  \leq \|\kbf\|^{2r}$. The real question is how to bound $|\mathrm{Bad}_j(\kbf)|$ for each $j=1,\ldots, n$.
The recent work of the second author \cite{Xu18} proves a set of bounds that are perfectly suited for our purposes. 
This takes the form of a ``stratification,'' in the spirit of \cite[Prop. 1.0]{Fou00}, \cite[Prop. 3.2, Thm. 3.3]{Lau00},  \cite[Thm. 1.1 and 1.2]{FouKat01}.  

To give the flavor of this stratification, we state here a special case in dimension $n=2$ (see  \S \ref{sec_Xu}  and in particular Theorem \ref{thm_Xu} for the full setting).  
Let $\kbf=(k,k,\ldots,k)$ be a fixed tuple in $\Z^n$, with $k\geq 1$. We will consider tuples $\xbf \in \Z^n$ that lie in the box $\xbf \in (\zerobf,\kbf]$, and more generally, a collection $\{\xbf\}  = \{\xbf\up1,\dots, \xbf\up{2r}\}\in(\zerobf,\kbf]^{2r}$ of $2r$ such $n$-tuples.
\begin{thm}\label{thm_Xu_2}
Let $n=2$ and let $r,\Del,D \geq 1$ be fixed.  
There exists a constant $C=C(n,r,D)$ and a constant $C'' = C''(n,r,\Del,D)$ such that the following holds. For any prime $q$, for any non-principal multiplicative Dirichlet character $\chi$ of order $\Delta$ modulo $q$, and 
for any $F \in \Z[x_1,x_2]$ that is $(\Del,q)$-admissible, for every tuple $\kbf = (k,k)$,
\[ \# \left\{ \{\xbf\up1,\dots, \xbf\up{2r}\}\in(\zerobf,\kbf]^{2r}
  : \left|  \sum_{\mbf \modd{q}}
\chi(F_{\{\xbf\}}(\mbf))
\right|  > C q^{(n+j-1)/2} \right\} \leq C''
 \begin{cases}
		k^{4r} & \text{if  $j=0$}\\
		k^{3r+1} & \text{if $j=1$}\\
		k^{2r} & \text{if  $j=2$}.
		\end{cases}\]
\end{thm}

We can interpret this as follows:  for $n=2$, the trivial bound for the number of collections $\{\xbf\}$ in the box $(\zerobf,\kbf]^{2r} \subset \Z^{2nr}$ is $k^{4r}$. 
When $j=1$, we see that at most $O(k^{3r+1})$ collections can violate square-root cancellation, i.e. as soon as $r \geq 2$, generically square-root cancellation holds. This stratification, in its  general formulation  (Theorem \ref{thm_Xu}) is the key input which allows us to prove the Burgess bound for all dimensions $n \geq 2$. (See furthermore \S \ref{sec_strat_1} for a demonstration of why the full stratification is useful.)

We stated  Theorem \ref{thm_Xu_2} for simplicity in the case where the length $\kbf $ has identical values in each coordinate. In our argument, we must instead allow $\kbf = (k_1,\ldots, k_n)$ with  $k_i$ varying independently; in particular this arises in the step when we deduce a bound for the maximal moment in (\ref{SFNH_D4_intro}) from the non-maximal moment in (\ref{S_heuristic2_intro}). This raises another difficulty in the multi-dimensional setting, which we now outline.
 
 Xu's stratification (Theorem \ref{thm_Xu_original}, Theorem \ref{thm_F_power}) shows that for each $j=1,\ldots, n$, the collections $\{ \xbf^{(1)},\ldots,\xbf^{(2r)} \}$ counted by $\mathrm{Bad}_j(\kbf)$ lie on a certain subscheme over $\F_q$ with a certain codimension. 
  In general, fix a dimension $R$,  let $X \subset \Abb_{\F_q}^R$ be a subscheme of codimension $\varpi$ and let $U$ be the sum of the degrees of its irreducible components. For subsets $B_i \subset \F_q$, define the ``box'' $B = \prod_{i=1}^R B_i \subset \Abb^R_{\F_q}$. We require a bound for 
\beq\label{IXB_dfn}
 I_X(B) := \# ( X(\F_q) \intersect \prod_{i=1}^R B_i),
 \eeq
which depends only on the codimension of $X$ and the degree of the irreducible components of $X$. A trivial bound, best possible if $\codim X=0$, is $I_X(B) \leq \|B\| = \prod_{i=1}^R |B_i|$. We need to improve on this when $\codim X \geq 1$.

To gain an intuition, consider the case of $R=2$ and $X$ of codimension 1. A naive hope might be that $I_X(B) \leq  \|B\|^{1/2} = \|B\| \cdot |B_1|^{-1/2} |B_2|^{-1/2}$. But this need not be true. Supposing for example that $|B_1| \approx 1$ is very small while $|B_2|$ is very large, it could be that $X$ of codimension 1 lies along the subset $B_2$ of $B$, thus leading to the bound $I_X(B) \gg |B_2| = \|B\| \cdot |B_1|^{-1} \approx \|B\|$. Thus in general,   in estimating $I_X(B)$ we can only expect to save factors corresponding to the shortest sides of the box $B$. 
Concretely, for $X$ of codimension $\varpi$, we use the fact that if 
\beq\label{ordering_B}
1 \leq |B_1| \leq |B_2| \leq \cdots \leq |B_R| < \infty
\eeq then by \cite[Lemma 1.7]{Xu18},
\beq\label{IXB}
 I_X(B) \leq U \|B\| \cdot |B_1|^{-1} |B_2|^{-1} \cdots |B_{\varpi}|^{-1},
 \eeq
 where $U$ is the degree of $X$.
In our setting, we will apply such bounds in \S \ref{sec_MR} when using a Menchov-Rademacher technique to deduce a bound for a maximal moment from a non-maximal moment. Here, in order to guarantee the ordering (\ref{ordering_B}) we must permute variables in a delicate argument, and apply rearrangement inequalities in order to conclude.

\section{Preliminaries}\label{sec_prelim}

\subsection{Power-free conditions}\label{sec_powerfree}
We say that $F \in \Z[x_1,\ldots, x_n]$ is $d$-th power-free if each non-constant irreducible factor of $F$ over $\Z$  (or equivalently over $\Q$, by Gauss's lemma) appears with multiplicity strictly less than $d$.  In general, given a field $k$, we say that $F\in k[x_1,\ldots, x_n]$ is $d$-th power-free over $k$ when $F=cF_1^{a_1} F_2^{a_2} \cdots F_\ell^{a_\ell}$ with $c \in k^\times$, all $a_i < d$ and all $F_i \in k[x_1,\ldots, x_n]$ are  irreducible and pairwise non-associate.
(We say $G,G' \in k^s[x_1,\ldots, x_n]$ are non-associate if there is no unit in $k^s[x_1,\ldots, x_n]$ such that $G=cG'$.) 
 To be precise, we recall that this property may be specified equivalently over a field $k$ or the separable closure of $k$:

\begin{lemma}\label{lemma_d_free}
Let $k$ be a field and $k^s$ its separable algebraic closure, so $k^s = \overline{k}$ if $k$ is perfect, and in particular if $k$ is finite.
Then for any $F  \in k[x_1,\ldots, x_n]$, $F$ is a perfect $d$-th power over $k$ (up to a nonzero multiplicative constant) if and only if $F$ is a perfect $d$-th power over $k^s$.
Similarly, $F$ is $d$-th power-free over $k$ if and only if $F$ is $d$-th power-free over $k^s$.
\end{lemma}

\begin{proof}
We begin with the second claim. Certainly if $F$ is $d$-th power-free over $k^s$ then it  is over $k$.
For the other direction, write $F=cF_1^{a_1} F_2^{a_2} \cdots F_\ell^{a_\ell}$ with $c \in k^\times$, all $a_i < d$ and all $F_i \in k[x_1,\ldots, x_n]$  irreducible and pairwise non-associate.  For each such $F_i$, we recall from \cite[Lemma 3.15 part (2)]{Xu18} that the fact that $F_i$ is irreducible over $k$ implies that $F_i$ is square-free as a polynomial in $k^s[x_1,\ldots, x_n]$. Thus upon factoring $F_i$ over $k^s$ we have $F_i = G_{i,1} G_{i,2} \cdots G_{i,b_i}$ in which each $G_{i,j}$ is irreducible in $k^s$ and as $j$ varies the $G_{i,j}$ are pairwise non-associate. 
Thus the factorization of $F$ over $k^s$ is  
\[ c (G_{1,1}^{a_1} G_{1,2}^{a_1} \cdots G_{1,b_1}^{a_1})  (G_{2,1}^{a_2} G_{2,2}^{a_2} \cdots G_{2,b_2}^{a_2})  \cdots
	 (F_{\ell,1}^{a_\ell} F_{\ell,2}^{a_\ell} \cdots F_{\ell,b_\ell}^{a_\ell}) .\]
	 Next we recall from  \cite[Lemma 3.15 part (3)]{Xu18} that if $F_i,F_{i'} \in k[x_1,\ldots,x_n]$ are non-associate irreducible polynomials, then $F_{i}$ and $F_{i'}$ have no common factors in $k^s[x_1,\ldots, x_n]$. From this we conclude that $G_{i,j}$ and $G_{i',j'}$ are non-associate when $(i,j) \neq (i',j')$.  Thus as $i$ and $j$ vary the $G_{i,j}$ are all pairwise non-associate over $k^s$, so that $F$ remains $d$-th power-free over $k^s$. 
	 
	 Finally, if $F$ is a perfect $d$-th power over $k$ then it also is over $k^s$. In the other direction, if $F$ is a perfect $d$-th power over $k^s$, then in the factorization above, all $a_i$ must be multiples of $d$, so that $F$ factors over $k$ as $F=cG^d$ with $G = F_1^{a_1/d} \cdots F_\ell^{a_\ell/d}$. 
\end{proof}

\subsection{Translation invariance conditions}
It is natural to impose on $F$  that it be appropriately nondegenerate, in the sense that it cannot be made independent of one or more variables. Indeed, if there exists a linear change of coordinates $\xbf \mapsto  \xbf A$ with $A \in \mathrm{GL}_n(\Z)$ such that $F(\xbf A) \in \Z[x_2,\ldots, x_n]$ then we would not expect $|S(F;\Nbf,\Hbf)|$ to obey bounds of the full $n$-dimensional strength that we will obtain.  We will require several equivalent formulations for the condition that $F$ is nondegenerate in this sense.

We recall  six equivalent statements about a polynomial $F \in \Z[x_1,\ldots, x_n]$ having the property that it can be made independent of one of the indeterminates by a linear change of coordinates over $\Z$.

\begin{lemma}[{Lemma 3.20 \cite{Xu18}}]\label{lemma_F_equiv}
Let $F \in \Z[x_1,\ldots, x_n]$  and let $\xbf = (x_1,\ldots, x_n)$ be the row vector of indeterminates. Then the following are equivalent. 
\begin{enumerate}
\item $F$ is invariant under some nontrivial translation in $\overline{\Q}^n$, i.e. there exists $0 \neq \mbf \in \Qbar^n$ such that $F(\xbf) \equiv F(\xbf+\mbf)$.
\item $F$ is invariant under some nontrivial translation in $\Z^n$, i.e. there exists $0 \neq \mbf \in \Z^n$ such that $F(\xbf) \equiv F(\xbf+\mbf)$.
\item $F$ can be made independent of one of the indeterminates by a linear change of coordinates, 
 i.e. there exists $A \in \mathrm{GL}_n(\Z)$ such that $F(\xbf A) \in \Z[x_2,\ldots, x_n]$.
 \item When viewed as a morphism $\Abb_\Z^n \maps \Abb_\Z^1$, $F$ factors through a linear map $\Abb_\Z^n \maps \Abb_\Z^{n-1}$, i.e. there exists an integral $n \times (n-1)$ matrix $B$ and $f \in \Z[x_2,\ldots, x_n]$ such that $F(\xbf) \equiv f(\xbf B)$.
 \item For almost all prime numbers $q$ (all but finitely many), the reduction of $F$ modulo $q$ is invariant under some nontrivial translation in $\overline{\F}_q^n$.
 \item For infinitely many prime numbers $q$, the reduction of $F$ modulo $q$ is invariant under some nontrivial translation in $\overline{\F}_q^n$.
 \end{enumerate}
\end{lemma}

We now also require an analogue of this over a field $k$.

\begin{lemma}\label{lemma_equiv_k}
Let $k$ be a perfect field. 
 Let $F \in k[x_1,\ldots, x_n]$ and suppose that $\deg F < \mathrm{char}\, k$ if $k$ is of positive characteristic. Let $\xbf$ be the row vector $(x_1,\ldots, x_n)$ of indeterminates. Then the following are equivalent. 
\begin{enumerate}
\item $F$ is invariant under some nontrivial translation in $\overline{k}^n$, i.e. there exists $\zerobf \neq \mbf \in \overline{k}^n$ such that $F(\xbf) \equiv F(\xbf+\mbf)$.
\item $F$ is invariant under some nontrivial translation in $k^n,$ i.e. there exists $\zerobf \neq \mbf \in k^n$ such that $F(\xbf) \equiv F(\xbf+\mbf)$.
\item $F$ can be made independent of one of the indeterminates by a linear change of coordinates, i.e. there exists $A \in \mathrm{GL}_n(k)$  such that $F(\xbf A) \in k[x_2,\ldots, x_n]$.
\item When viewed as a morphism $\mathbb{A}_k^n \maps \mathbb{A}_k^1$, $F$ factors through a linear map 
$\mathbb{A}_k^n \maps \mathbb{A}_k^{n-1}$, i.e. there exists an $n \times (n-1)$ matrix $B$ with entries in $k$, and $f \in k[x_1,\ldots, x_n]$ such that $F(\xbf) \equiv f(\xbf B)$. 
\end{enumerate}
\end{lemma}
\begin{proof}[Proof of Lemma \ref{lemma_equiv_k}]
$(1) \implies (2)$: Let $\zerobf \neq \mbf = (m_1,\ldots, m_n) \in \overline{k}^n$ be such that $F(\xbf) \con F(\xbf + \mbf)$, and assume without loss of generality that $m_1 \neq 0$. By iteration, $F(\xbf + t\mbf) - F(\xbf) \con 0$ as a function of $\xbf$, for all $t \in \Z$, hence for all $t \in k_0$, the prime field inside $k$, if $\ch k>0$. We consider separately the case of characteristic zero: if $\ch k=0$ we directly conclude that $F(\xbf + t\mbf) - F(\xbf)=0$ as a polynomial in $\overline{k}[x_1,\ldots,x_n,t]$ since a nonzero polynomial cannot have infinitely many roots (namely all $t\in\Z$).  In the positive characteristic case, we learn that $t^{\ch k} - t$ divides  $F(\xbf + t\mbf) - F(\xbf)$ as polynomials in $\overline{k}[x_1,\ldots,x_n,t]$.  Under the assumption $\deg F < \ch k$, it therefore must be the case that 
\beq\label{FFq}
F(\xbf + t\mbf) - F(\xbf)=0
\eeq
 as a polynomial in $\overline{k}[x_1,\ldots,x_n,t]$. Now let $E$ be the field generated by $m_1,\ldots, m_n$ over $k$, and choose any $t \in E$ such that $\Tr_{E/k}(t) \in k \setminus \{0\}$;  such a $t$ is guaranteed as long as $E/k$ is separable, which holds because we assumed that $k$ is perfect. (For $\ubf \in E^n$ we will let $\Tr_{E/k}(\ubf) = (\Tr_{E/k}(u_1),\ldots, \Tr_{E/k}(u_n))$.) Then using the fact (\ref{FFq}) and $m_1 \neq 0$, we see that also $\Tr_{E/k}(t\mbf/m_1) \in k^n \setminus\{\zerobf\}$.  
 We will now observe that  $F(\xbf + \Tr_{E/k}(t\mbf/m_1)) \con F(\xbf)$, concluding the proof of (2).
Since $F$ has coefficients in $k$, then for any $\sig \in \Gal(E/k)$, we have $F(\xbf) \con F(\xbf + \sig(t\mbf/m_1))$.  
 Consequently, we have $F(\xbf) \con F(\xbf + \Tr_{E/k}(t\mbf/m_1))$, as desired.

$(2) \implies (3)$: Let $\zerobf \neq \mbf = (m_1,\ldots, m_n) \in k^n$ be such that $F(\xbf ) \con F(\xbf + \mbf)$; then proceeding as in the previous argument, this implies that $F(\xbf+t\mbf) \con F(\xbf)$ as polynomials in $k[x_1,\ldots, x_n,t]$. We will show that there exists $A \in \mathrm{GL}_n(k)$ such that $\mbf = (1,0,\ldots,0) A$. Once we have this matrix, we observe that $F (\xbf A) \con F(\xbf A + \mbf) \con F( (\xbf + (1,0,\ldots, 0))A)$ so that upon defining $G(\xbf) = F(\xbf A)$, we have that $G(\xbf)$ is invariant under translation $\xbf \mapsto \xbf + (1,0,\ldots, 0)$. Consequently, when  regarded as polynomial in $x_1$,
\beq\label{GminusG} 
G(x_1,x_2,\ldots, x_n)  - G(0,x_2,\ldots, x_n)
\eeq
has all integers as its roots, and hence all elements in $k_0$ as its roots if $\mathrm{char}\, k >0$. In the characteristic zero case, this implies that (\ref{GminusG}) is the zero polynomial in $x_1$, and hence $F(\xbf A) = G(\xbf) = G(0,x_2,\ldots, x_n) \in k[x_2,\ldots, x_n]$. If $\mathrm{char}\, k>0$, we learn that $x_1^{\mathrm{char}\, k} - x_1$ divides (\ref{GminusG}), but since $\deg F < \mathrm{char}\, k$, it must be the case that (\ref{GminusG}) is identically the zero polynomial over $k$. Hence as before we have $F(\xbf A) = G(\xbf) = G(0,x_2,\ldots, x_n) \in k[x_2,\ldots, x_n]$, concluding the proof.

Finally, we construct the matrix $A$. Note that $\mathrm{GL}_n(k)$ acts transitively on nonzero vectors in $k^n$, since any such vector is an element in a basis for $k^n$, and there exists a unique element in $\mathrm{GL}_n(k)$ mapping one ordered basis to another. Thus in particular there exists $A \in \mathrm{GL}_n(k)$ such that $\mbf =(1,0,\ldots,0)A$, as desired.

$(3) \implies (4)$: Suppose that $F(\xbf A) \con f(x_2,\ldots, x_n)$ for some $f \in k[x_2,\ldots, x_n]$, so $F(\xbf)\con F( (\xbf A^{-1})A) \con f( (\xbf A^{-1})_2,\ldots, (\xbf A^{-1})_n)$, where $A^{-1} \in \mathrm{GL}_n(k)$. Then it suffices to define $B$ to be the matrix constructed of the last $n-1$ columns of $A^{-1}$.

$(4) \implies (1)$: Now we suppose that there exists such an $n \times (n-1)$ matrix $B$ and $f \in k[x_2,\ldots, x_n]$ such that $F(\xbf) \con f(\xbf B)$. Since multiplication by $B$ is a linear map from $k^n$ to $k^{n-1}$, the nullspace of this map is nontrivial and hence there exists $\zerobf \neq \mbf \in k^n$ such that $\mbf B = \zerobf$; consequently 
\[ F(\xbf) \con f(\xbf B) \con f(\xbf B + \mbf B) \con f((\xbf + \mbf)B) \con F(\xbf + \mbf).\]
This implies (2), which certainly implies (1), concluding the proof of the lemma.
\end{proof}

\subsection{All but finitely many primes}
In the introduction we stated the  fact that forms that are ``admissible'' over $\Z$ are ``admissible'' over $\F_q$ for all but finitely many primes. The formal statement is here:
\begin{lemma}\label{lemma_F_ae_q}
Let $\Del \geq 1$ be fixed. Let $F \in \Z[x_1,\ldots, x_n]$ and suppose $F$ factors as $F=G^\Del H$ with $G,H \in \Z[x_1,\ldots, x_n]$ and $H$ being $\Del$-th power-free over $\Z$. Furthermore, assume that $H$ is nondegenerate over $\Z$, in the sense that there is no $A \in \mathrm{GL}_n(\Z)$ such that $H(\xbf A) \in \Z[x_2,\ldots, x_n]$. Then for all but finitely many primes $q$, $F$ is $(\Del,q)$-admissible.
\end{lemma}
\begin{proof}
For a fixed prime $q$, we reduce $F,G,H$ modulo $q$ to $f,g,h \in \F_q[x_1,\ldots, x_n]$.  
By Lemma 3.22 of [Xu18], since $H$ is $\Del$-th power-free over $\Z$ then for all but finitely many primes $q$,  $h$ is $\Delta$-th power-free over $\F_q$.  Letting $Q_1$ denote this finite set of exceptional primes, then for all $q \not\in Q_1$, we have that $f = g^\Del h$  with $h$ being $\Del$-th power-free over $\F_q$. 
As a consequence of Lemma \ref{lemma_F_equiv} part (6), since $H$ is nondegenerate over $\Z$, the reduction $h$ of $H$ modulo $q$ can be invariant under a nontrivial translation in $\overline{\F}_q^n$ only for finitely many primes $q$; we will call this exceptional set $Q_2$. Finally, let $Q_3$ denote the primes $q \leq \deg F$. We now proceed to consider the primes $q \notin Q_1 \union Q_2 \union Q_3;$ for such primes, $h$ is not invariant under any nontrivial translation in $\F_q^n$, and (3) in Lemma \ref{lemma_equiv_k} shows that $h$ cannot be made independent of any indeterminate by a linear change of variables in $\mathrm{GL}_n(\F_q)$. This proves the lemma.
(Here we ruled out the primes in $Q_3$ because we cite Lemma \ref{lemma_equiv_k}, but we note that the specific implications of this lemma that we employ here do not need the assumption $q > \deg F$.)
\end{proof}

\subsection{Permutations of variables}\label{sec_perm}
Within the Burgess method, we will use a variant of the Menchov-Rademacher method for deducing bounds for maximal moments from non-maximal moments. To carry out this argument in our setting, we will need to re-order the variables $x_1,x_2,\ldots, x_n$ in $F(\xbf)$ so that a corresponding tuple of parameters $(k_1,\ldots, k_n)$ satisfies the ordering $k_1 \leq k_2 \leq \cdots \leq k_n$. Thus we are led to consider forms resulting from $F$ when the variables are permuted.   For any permutation $\pi$ of $\{1,\ldots, n\}$, define the form $F_\pi(\Xbf)$ from the form $F(\Xbf)$ by setting
$F_\pi(X_1,\ldots, X_n) = F(X_{\pi(1)},\ldots, X_{\pi(n)}).$  
\begin{lemma}\label{lemma_perm}
Let $\Del \geq 1$ and a prime $q$ be fixed. Then given a $(\Del,q)$-admissible  form $F$, the form $F_\pi$ is $(\Del,q)$-admissible for all permutations $\pi$ on a set of $n$ elements.
\end{lemma}
\begin{proof}
Letting $f$ denote the reduction of $F$ modulo $q$, we write $f=g^\Del h$ with $h$ being $\Del$-th power-free over $\F_q$; then correspondingly for the permuted versions, if $f_\pi$ denotes the reduction of $F_\pi$ modulo $q$ then $f_\pi = (g_\pi)^\Del h_\pi$ with $h_\pi$ being $\Del$-th power-free over $\F_q$. Moreover, $h$ can be made independent of at least one variable after a $\mathrm{GL}_n(\F_q)$ change of variable if and only if $h_\pi$ can.
\end{proof}

\section{The stratification of complete character sums}\label{sec_Xu}
In this section we recall the stratification of complete character sums proved by the second author in \cite{Xu18}  and show how to deduce a slightly stronger formulation that we believe will be of independent interest, as well as being useful in this paper. In our presentation, we will replace the dimension $r$ in the original work by the dimension $2r$ in each instance, but the content of this section would apply in an analogous way   for any dimension $r$ (odd or even).
 
\subsection{The stratification obtained by Xu}
 We first recall the statement of \cite[Theorem 1.1 and Corollary 1.8]{Xu18} in our setting of dimension $2r$. For each fixed $n,r \geq 1$, we define a set of parameters $\theta_j$ for $1 \leq j \leq n$, as follows:
\beq\label{theta_j_dfn}
 \theta_j = \begin{cases}
 	0 &j=0 \\
	 j\lfloor (r-1)/(n-1) \rfloor, &j=1, \ldots, n-2\\
	  r-1, &j=n-1, \\
 	nr, &j=n.
\end{cases}
\eeq
Note that this differs superficially from the definition of $\theta_j$ in \cite[p. 2]{Xu18}: we are working with dimension $2r$ in place of $r$ and the floor function  results in slightly different formulas. (Precisely,  from Xu's work  we may take 
\beq\label{theta1_dfn}
 \theta_1=\lfloor (2r-1)/(2n-2) \rfloor = \lfloor (r-1)/(n-1) \rfloor
\eeq
and then we set $\theta_j = j\theta_1$ for $1 \leq j \leq n-1$. In particular for  $j=n-1$, 
\[ \theta_{n-1} = \lfloor (2r-1)/2\rfloor = r-1.\]
In fact, Xu's original theorem allows a slightly larger value of $\theta_j \geq j \theta_1$; it is later apparent in (\ref{sum_j}) that this slightly larger choice, leading to a slightly stronger version of Theorem \ref{thm_Xu_original}, would not significantly improve our current application.)

 \begin{letterthm}[{Theorem 1.1 of \cite{Xu18}}]\label{thm_Xu_original}
 Let integers $n,r,\Del,D \geq 1$  be fixed.  There exist integers $C=C(n,r,D) \geq 1$ and $C' = C'(n,r,\Del,D) \geq 1$  and a finite set $\Scal = \Scal(n, r, \Del,D)$ of primes  such that the following holds. 
 
 Let $\kappa$ denote a finite field with algebraic closure denoted by $\overline{\kappa}$. For each $1 \leq i \leq 2r$, let $\chi_i: \kappa^\times \maps \C^\times$ be a multiplicative character (extended to $\kappa$ by setting $\chi_i(0)=0)$ and assume that $d_i := \ord(\chi_i) | \Del >0$. Let $F_i \in \kappa(x_1,\ldots, x_n)$ be a $d_i$-th power-free rational function of degree at most $D$ and assume that 
 \beq\label{TFi_dfn}
  T_{F_i} := \{ \mbf \in \overline{\kappa}^n : F_i (\xbf) \con F_i(\xbf+\mbf) \} 
  \eeq
 is finite for each $1 \leq i \leq 2r$. Then upon defining
 \beq\label{S_thm_dfn}
  S(\xbf^{(1)}, \ldots, \xbf^{(2r)}) := \sum_{\mbf \in \kappa^n} \prod_{i=1}^{2r} \chi_i (F_i(\mbf+\xbf^{(i)})),
  \eeq
 we have that whenever $\mathrm{char}\, \kappa \not\in \Scal$, there exist subschemes 
 \[ \mathbb{A}_\kappa^{2nr} = X_0 \supset X_1  \supset X_2 \supset \cdots \supset X_n \]
 such that for each $ 1 \leq j \leq n$,
 \begin{enumerate}
 \item the sum of the degrees of irreducible components of $X_j$ is at most $C'$;
 \item $\dim X_j \leq 2nr - \theta_j$ with $\theta_j$ defined as in (\ref{theta_j_dfn});
\item  for all $(\xbf^{(1)}, \ldots, \xbf^{(2r)}) \in \mathbb{A}^{2nr}(\kappa) \setminus X_j (\kappa),$ 
 \[ |S(\xbf^{(1)}, \ldots, \xbf^{(2r)}) | \leq C (\# \kappa)^{(n+j-1)/2}. \]
 \end{enumerate}
\end{letterthm}

\begin{remark}
Note that while the subschemes $X_j$ may depend on the $F_i$, the parameters $C, C', \theta_j$ depend only on the maximum degree $D$ of the $F_i$. 
The constant $C=C(n,r,D)$ noted above has the value $(4r(D+1)+1)^n$ as computed in \cite[Thm. 11]{Kat01} (see \cite[p. 3]{Xu18}, again recalling that we have $2r$ in place of $r$). The constant $C'$ (and later $C''$) only has dependence on $n, r, D, \Del$ but is not explicitly determined; see \cite[p. 3]{Xu18} and the discussion in Corollary \ref{cor_S_empty}.

 Note also that we did not explicitly exclude cases in which for some $1 \leq i \leq 2r$, $\chi_i$ is the principal character. However,  if $\chi_i$ is principal then $d_i = 1$ and hence $F_i \in \kappa(x_1,\ldots, x_n)$ must be a constant function. This implies that $T_{F_i}$ is infinite, thus excluding this possibility from the theorem. This remark applies to Corollary \ref{cor_Xu_original}, below, as well.
\end{remark}

Next, we must convert Theorem \ref{thm_Xu_original} into a count for the number of points $(\xbf^{(1)}, \ldots, \xbf^{(2r)}) \in \mathbb{A}^{2nr}(\kappa) $ (later denoted as collections $\{ \xbf^{(1)}, \ldots, \xbf^{(2r)} \}$) such that a given upper bound for $|S(\xbf^{(1)}, \ldots, \xbf^{(2r)})|$ holds. It  is convenient to define for any sequence $1 \leq k_1 \leq k_2 \leq \cdots \leq k_n<\infty$  the function
 \beq\label{Bnrj_dfn}
 B_{n,r,j}(\kbf) =  B_{n,r,j}(k_1, \ldots, k_n) 
 	=\begin{cases}
 1 & j=0 \\
 k_1^{\theta_j} = k_1^{j\lfloor \frac{r-1}{n-1} \rfloor}  & j=1,\ldots, n-2 \\
 k_1^{\theta_{n-1}}= k_1^{r-1}  & j=n-1 \\
  (k_1 \cdots k_{n/2})^{2r}  & j=n,  \text{$n$ even} \\
    (k_1 \cdots k_{(n-1)/2})^{2r}k_{(n+1)/2}^{r}  & j=n, \text{$n$ odd} . 
 \end{cases}
 \eeq
 (For readers with \cite{Xu18} in hand, to aid comparison to Xu's original notation involving parameters denoted by $n_0,\eta$, we make the simple observation that for $j=n$ and $\theta_j=nr$,  if $n$ is even we  write $nr = (n/2) 2r$ (with $n_0 = n/2$ and $\eta=0$) and when $n$ is odd we write $nr = ((n-1)/2) \cdot 2r + r$ (with $n_0 = (n-1)/2$ and $\eta=r$), so that Xu's original notation   results in the expression given for $B_{n,r,j}$ in the case $j=n$, and with ambient dimension $2r$.)

\begin{lettercor}[{Corollary 1.8 of \cite{Xu18}}]\label{cor_Xu_original}
 Let integers $n,r,\Del,D \geq 1$  be fixed.  There exist integers $C=C(n,r,D) \geq 1$ and $C' = C'(n,r,\Del,D) \geq 1$  and a finite set $\Scal = \Scal(n, r, \Del,D)$ of primes  such that the following holds. 
 
  Let $\kappa$ denote a finite field with algebraic closure denoted by $\overline{\kappa}$. 
 For each $1 \leq i \leq 2r$, let $\chi_i: \kappa^\times \maps \C^\times$ be a   multiplicative character (extended to $\kappa$ by setting $\chi_i(0)=0)$ and assume that $d_i := \ord(\chi_i) | \Del >0$. Let $F_i \in \kappa(x_1,\ldots, x_n)$ be a $d_i$-th power-free rational function of degree at most $D$ and assume that
$ T_{F_i} $ as defined in (\ref{TFi_dfn})
 is finite for each $1 \leq i \leq 2r$. 
 
 Let $\{ B_i \}_{i=1}^n$ be subsets of $\kappa$ such that $1 \leq |B_1| \leq |B_2| \leq \cdots \leq |B_n| < \infty$, and define $B = \prod_{i=1}^n B_i \subset \kappa^n$. Then whenever $\mathrm{char}\, \kappa \notin \Scal$, for each $1 \leq j \leq n$,
\beq\label{cor_B_bound}
 \# \{ (\xbf^{(1)}, \ldots, \xbf^{(2r)}) \in B^{2r} : |S(\xbf^{(1)}, \ldots, \xbf^{(2r)})| > C (\# \kappa)^{(n+j-1)/2} \} \leq C' \|B\|^{2r} B_{n,r,j}(|B_1|, \ldots, |B_n|)^{-1}.
\eeq
\end{lettercor}

\subsection{The stratification in our setting}
Now we state the new versions of Theorem \ref{thm_Xu_original} and Corollary \ref{cor_Xu_original} that we use in this paper.  
\begin{thm}\label{thm_F_power}
Let $\kappa = \F_q$, $q$ prime. The result of Theorem \ref{thm_Xu_original} holds if we replace the hypothesis that for each $1 \leq i \leq 2r$, the form $F_i$ is $d_i$-th power-free and $T_{F_i}$ is finite,  by  the weaker hypothesis that for each $1 \leq i \leq 2r$,  the form $F_i$ is $(d_i,q)$-admissible and that for each $1 \leq i \leq 2r$, $\deg F_i < q$.
\end{thm}

\begin{cor}\label{thm_F_power_cor} 
Let $\kappa = \F_q$, $q$ prime. The result of Corollary \ref{cor_Xu_original} holds if we replace the hypothesis that for each $1 \leq i \leq 2r$, the form $F_i$ is $d_i$-th power-free and $T_{F_i}$ is finite,  by  the weaker hypothesis that for each $1 \leq i \leq 2r$,  the form $F_i$ is $(d_i,q)$-admissible.
\end{cor}

 In the above corollary, we are able to omit the condition $\deg F_i < q$ seen in Theorem \ref{thm_F_power} by possibly enlarging $C'$; see (\ref{bigD}). Through similar considerations, we can remove consideration of the set $\Scal$, as we record here:

\begin{cor}\label{cor_S_empty}
In addition, given $n,r,\Del,D$ in either   Corollary \ref{cor_Xu_original} or Corollary \ref{thm_F_power_cor}, we may take the set $\Scal = \Scal(n,r,\Del,D)$ to be the empty set, 
at the expense of replacing $C' = C'(n,r,\Del,D)$ by a possibly larger constant $C'' = C''(n,r,\Del,D).$
\end{cor}

Within the Burgess argument, we will consider the sum 
\[\sum_{\mbf \modd{q}} \chi(F_{\{\xbf\}}(\mbf)) = \sum_{\mbf \modd{q}} \prod_{i=1}^{2r} \chi_i(F(\Xbf + \xbf^{(i)}))
\]
in which we have fixed a multiplicative Dirichlet character $\chi$ of order $\Del$ modulo a prime $q$ and then set $\chi_i = \chi$ if $i$ is odd and $\chi_i = \overline{\chi}$ if $i$ is even. This is then clearly of the form $S(\xbf^{(1)},\ldots, \xbf^{(2r)})$ as defined in (\ref{S_thm_dfn}), with the choice that all the $F_i$ are equal to our fixed  $F \in \F_q[x_1,\ldots, x_n]$ of degree $D$. For later reference, we record the following immediate consequence of  Theorem \ref{thm_F_power} and Corollary \ref{cor_S_empty} (choosing the subset $B_i$ of $\kappa$ to be $(0,k_i]$ in each instance).
\begin{thm}\label{thm_Xu}
Let integers $n, r, \Del, D \geq 1$ be fixed.  Then there exist constants $C=C(n,r,D)$ and $C'' = C''(n,r,\Del,D)$ such that the following holds.

Fix a prime $q$ and let $\chi$ be a  multiplicative Dirichlet character of order $\Del$ modulo  $q$.  Let $F \in \Z[x_1,\ldots, x_n]$ be $(\Del,q)$-admissible and define $F_{\{\xbf\}} (\Xbf)$ accordingly as in (\ref{dfn_Gpoly}).
Then for every  $1\le j\le n$, for every tuple $\kbf = (k_1,\ldots, k_n) \in \Z^n$ with $1 \leq k_1 \leq k_2 \leq \cdots \leq k_n \leq q$, 
\beq\label{Xu_ineq}
 \# \left\{ (\xbf^{(1)}, \ldots, \xbf^{(2r)})
\in (\zerobf,\kbf]^{2r} : \left|  \sum_{\mbf \modd{q}}
\chi(F_{\{\xbf\}}(\mbf))
\right|  > C q^{(n+j-1)/2} \right\}
 \leq C''  \| \kbf \|^{2r} B_{n,r,j}(\kbf)^{-1},
 \eeq
 in which $B_{n,r,j}(\kbf)$ is defined as in (\ref{Bnrj_dfn}). 
\end{thm}

\begin{remark}
Note that the trivial upper bound in (\ref{Xu_ineq}) is $\|\kbf\|^{2r}.$ The fundamental consequence of Theorem \ref{thm_Xu} is that it shows that \emph{generically} among $\{\xbf \} \in (\zerobf,\kbf]^{2r}$, square-root cancellation holds, as soon as $r$ is sufficiently large relative to $n$.
Precisely, as soon as $r \geq n$, so that the exponent $\lfloor (r-1)/(n-1) \rfloor$ appearing in $B_{n,r,1}(\kbf)$ is strictly positive, the number of $\{\xbf \} \in (\zerobf,\kbf]^{2r}$ such that square-root cancellation is violated is $O(\|\kbf\|^{2r} k_1^{-1})$, which suffices for our claim, as long as $k_1$ is at least a positive power of $\|\kbf\|$.

Also, to aid in understanding the role of the function $B_{n,r,j}(\kbf)$ in this result, we note that the bound (\ref{Xu_ineq}) is in the format of (\ref{IXB}) with the choice $R=2nr$ and with the $R$-dimensional box being 
\[ (0,k_1]\times \cdots \times (0,k_1] \times \cdots \times (0,k_n]\times \cdots \times (0,k_n] ,\]
in which each factor $(0,k_i]$ appears $2r$ times. Thus when $\theta_j \leq 2r$ (this holds for $j\leq n-1$), we only save factors of $k_1$. In the final case $j=n$ when $\theta_n  =nr$, we save some factors of   $k_i$ for $1 \leq i \leq \lceil n/2 \rceil$ as well.
This leads to the definition of $B_{n,r,j}(\kbf)$.
Finally, we remark for later reference that by construction, under the hypotheses of the theorem,
\beq\label{Bnrj_bound_below}
\|\kbf\|^{2r} B_{n,r,j}(\kbf)^{-1} \geq 1.
\eeq

\end{remark}

\begin{remark} \label{remark_optimal}
 Conjecturally, one might hope to improve the result of Theorem \ref{thm_Xu} by proving that one can take larger values for the codimension $\theta_j$. (Precise implications may be found in Section \ref{sec_improvements}, where we show that even the conjecturally best possible values for the codimension do not significantly change our main result.)
For comparison, in the most extreme case, it is not hard to see that we must have $\theta_n \leq nr$, and hence certainly must also have at most $\theta_j  \leq nr$ for all $j \leq n$. For recall from Theorem \ref{thm_Xu_original} that $X_n$ is a subscheme of $\mathbb{A}(\kappa)^{2nr}$ such that for all 
$(\xbf^{(1)}, \ldots, \xbf^{(2r)}) \in \mathbb{A}^{2nr}(\kappa) \setminus X_n (\kappa),$ 
$ |S(\xbf^{(1)}, \ldots, \xbf^{(2r)}) | \leq C (\# \kappa)^{n-1/2}. $ In fact, Xu's paper shows the stronger result that for all finite extensions $k/\kappa$ and $(\xbf^{(1)}, \ldots, \xbf^{(2r)}) \in \mathbb{A}^{2nr}(k) \setminus X_n (k),$ it holds that $ |S_k(\xbf^{(1)}, \ldots, \xbf^{(2r)}) | \leq C (\# k)^{n-1/2};$ here $S_k$ is defined analogously to $S$ but summing over $\mbf \in k^n$ and with $\chi_i$ replaced by $\chi_i(N_{k/\kappa}(\cdot))$.
That is, all $(\xbf^{(1)}, \ldots, \xbf^{(2r)})$ such that $ |S_k(\xbf^{(1)}, \ldots, \xbf^{(2r)}) | > C (\# k)^{n-1/2}$ must lie in $X_n(k)$. Then we claim that $\dim X_n \geq nr$ (and consequently $\theta_n \leq nr$). To see this, we consider any tuple $(\xbf^{(1)}, \ldots, \xbf^{(2r)})$ for which $\xbf^{(j)}=\xbf^{(j+r)}$ for $j=1,\ldots, r$. There are $(\# k)^{nr}$ such tuples, and each of them has the property that $ |S_k(\xbf^{(1)}, \ldots, \xbf^{(2r)}) |= (\# k)^{n}$.
This is $> C(\# k)^{n-1/2}$ if $\# k$ is sufficiently large (and we can choose it to be).
\end{remark}

\subsection{Deduction of the corollaries}
Corollary \ref{thm_F_power_cor} follows from Theorem \ref{thm_F_power} in an identical fashion to how Corollary \ref{cor_Xu_original} follows from Theorem \ref{thm_Xu_original} and we do not repeat the proof here. We only note that in Corollary \ref{thm_F_power_cor} we no longer need to assume that $\deg F_i <q$. For indeed, suppose that for some $1 \leq i \leq 2r$ we have $\deg F_i \geq \mathrm{char}\, \kappa$. Then 
we note that trivially 
\beq\label{bigD}
|S(\xbf^{(1)}, \ldots, \xbf^{(2r)})| \leq (\# \kappa)^{2nr}  = q^{2nr}  \leq D^{2nr}
\eeq
 for all $(\xbf^{(1)}, \ldots, \xbf^{(2r)}) \in \mathbb{A}^{2nr}(\kappa)$, so that upon enlarging $C'$ if necessary so that $C'  \geq  D^{2nr}$ the results of Corollary \ref{thm_F_power_cor} hold (here we also use the fact (\ref{Bnrj_bound_below})).

 To  obtain Corollary \ref{cor_S_empty} in which formally  $\Scal = \emptyset$, we note that for any $q \in \Scal(n,r,\Del,D)$, we may write the trivial upper bound $C' (\# B)^{2r} \leq C' q^{2nr}$ on the right-hand side of (\ref{cor_B_bound}).   Thus in order to state a version of Corollary \ref{cor_Xu_original} or Corollary \ref{thm_F_power_cor} with $\Scal= \emptyset$, we simply replace $C'$ in the statement of the corollary by 
$C''(n,r,\Del,D) = \max \{ C', q^{2nr} : q \in \Scal(n,r,\Del,D) \}.$

\subsection{Proof of Theorem \ref{thm_F_power} }
Theorem \ref{thm_F_power} follows from a small modification inside the proof of Theorem \ref{thm_Xu_original}   in \cite{Xu18}. To be clear, we will state exactly the change that is made (recalling that in our setting we use $2r$ where \cite{Xu18} uses $r$; our modifications would of course work for any dimension $r$). 

The main idea is that even if there is an $ i \in \{1,\ldots, 2r\}$ such that $F_i$ is not $d_i$-th power-free, we can write $F_i = G_i^{d_i} \widetilde{F_i}$ in which $\widetilde{F_i}$ is $d_i$-th power-free, and has $T_{\tilde{F_i}}$ finite, under the assumption that $F_i$ is $(d_i,q)$-admissible, and then at a key moment in the proof we work with $\widetilde{F_i}$ instead of $F_i$.
To be precise, recall that for any  $(\xbf^{(1)}, \ldots, \xbf^{(2r)}) \in \kappa^{2nr}$,
\beq\label{S_sum_dfn'}
  S(\xbf^{(1)}, \ldots, \xbf^{(2r)}) := \sum_{\mbf \in \kappa^n} \prod_{i=1}^{2r} \chi_i (F_i(\mbf+\xbf^{(i)})).
  \eeq
We also define for any tuple  $(\mbf^{(1)}, \ldots, \mbf^{(2s)}) \in \kappa^{2ns}$ and each $1 \leq i \leq 2r$ the function 
\[ T_i (\mbf^{(1)}, \ldots, \mbf^{(2s)}) = \sum_{\xbf \in \kappa^n}   \chi_i (F_{i,\{\mbf\}} (\xbf) ),\]
in which 
\[F_{i,\{\mbf\}} (\xbf) :=  \prod_{j=1}^{s} F_i(\mbf^{(j)} + \xbf) \prod_{j=s+1}^{2s} F_i(\mbf^{(j)} + \xbf)^{d_i-1} .\]
 The proof of Theorem \ref{thm_Xu_original} (Theorem 1.1 in \cite{Xu18}) and hence of Corollary \ref{cor_Xu_original} (Theorem 1.8 of \cite{Xu18}) relies on four ingredients.
 
  (I) The first ingredient \cite[Prop. 1.2]{Xu18} is an identity between $2s$-moments of the sums $S$ with $2r$-multilinear averages of the sums $T_i$, namely
 \beq\label{ST}
  \sum_{(\xbf^{(1)}, \ldots, \xbf^{(2r)}) \in \kappa^{2nr}} |S(\xbf^{(1)}, \ldots, \xbf^{(2r)})|^{2s}
 	 = \sum_{(\mbf^{(1)}, \ldots, \mbf^{(2s)}) \in \kappa^{2ns}} \prod_{i=1}^{2r} T_i (\mbf^{(1)}, \ldots, \mbf^{(2s)}) .
	 \eeq
We will refer to the left-hand side as the moment $M_{\kappa}(r,s)$; it has a natural generalization to a moment $M_k(r,s)$ defined in an appropriately analogous manner over any finite extension $k/\kappa$, with $\chi_i (\cdot)$ replaced by $\chi_i ( N_{k/\kappa} (\cdot))$. 

(II)	The second  ingredient \cite[Prop. 1.5]{Xu18} relates the moments $M_k(r,s)$ for finite extensions $k/\kappa$ to the dimension of the subschemes $X_j$.  

(III) The third   ingredient \cite[Prop. 1.6 (a)]{Xu18} is an upper bound of $O((\#k)^{ns})$ (that is, of square-root strength) for the number of tuples $(\mbf^{(1)}, \ldots, \mbf^{(2s)}) \in k^{2ns}$ such that 
$F_{i,\{\mbf\}} (\xbf) $
is a perfect $d_i$-th power in $\overline{k}(x_1,\ldots, x_n)$.

(IV) The fourth  ingredient \cite[Prop. 1.6(b)]{Xu18} is an application of the Weil bound to save one factor of $(\# k)^{1/2}$ off the trivial bound $(\# k)^n$ for an $n$-dimensional character sum, which is a generalization of $T_i$ in an extension $k/\kappa$. Precisely, it is the statement that uniformly in finite extensions $k/\kappa$ and tuples $(\mbf^{(1)}, \ldots, \mbf^{(2s)}) \in k^{2ns}$, if $F_{i,\{\mbf\}} \in k(x_1,\ldots,x_n)$ is not a perfect $d_i$-th power in $\overline{k}(x_1,\ldots, x_n)$, then 
\beq\label{N_bounds}
 \sum_{\xbf \in k^n} \chi_i (\mathrm{N}_{k/\kappa} (F_{i,\{\mbf\}}(\xbf)) = O( (\# k)^{n-1/2}).
 \eeq

These four ingredients are applied in a bootstrapping process. 
The general philosophy is that a weak bound with very few exceptions can be bootstrapped   into a stronger bound with possibly more exceptions. More precisely, ingredients (III) and (IV) are the initial input, showing that a small savings holds for the sums $T_i$, aside from possibly $O( (\# k)^{ns})$ many (that is, square-root many) exceptional choices of $(\mbf^{(1)}, \ldots, \mbf^{(2s)}) \in k^{2ns}$. For the possible  exceptional choices, a trivial upper bound of $O((\# k)^n)$ is applied in place of (\ref{N_bounds}).

This input step provides a savings, on average, for the sums $T_i$ on the multilinear right-hand side of the identity (\ref{ST}) in ingredient (I) and hence for the moment of $S$ on the left-hand side of (I).  Ingredient (II) then expresses this savings on the moment of $S$ as a stratification in terms of a lower bound on $\mathrm{codim} X_j$ for each $j$. This result holds uniformly for sums $S$ of the shape (\ref{S_sum_dfn'}). Since each $T_i$ is also a sum of this shape (with $s,\chi_i, F_i$ defined appropriately), the resulting bound for sums $S$ can be applied to each $T_i$, yielding an improvement over the initial savings for $T_i$. This argument then bootstraps to prove the final result. 

With this outline in hand, we may now briefly verify Theorem \ref{thm_F_power}. The only point at which this argument utilized the assumption that each $F_i$ is $d_i$-th power-free and has $T_{F_i}$ finite was  at the initial input to the bootstrapping, when ingredient (III) was used once (see \cite[\S 2.3]{Xu18}). Thus all we must do is show that this step, namely \cite[Prop. 1.6(a)]{Xu18} can be proved under the alternate assumption that each $F_i$ is $(d_i,q)$-admissible and has $\deg F_i< q$. 
We will replace \cite[Prop. 1.6(a)]{Xu18}   with the following proposition.
\begin{prop}\label{cor_count'}
Let $\kappa = \F_q$ with $q$ prime and fix a rational function $F \in \kappa(x_1,\ldots, x_n)$ that is $(d,q)$-admissible, with $\deg F<q$. Fix $s \in \N$. Then for each finite extension $k/\kappa$, the number of   $(\mbf^{(1)}, \ldots, \mbf^{(2s)}) \in k^{2ns}$ such that 
$F_{\{\mbf\}}(\xbf)$ is a perfect $d$-th power over $\overline{\kappa}$ is at most $ O( (\#k)^{ns})$.
\end{prop}

Once this has been proved, this replaces \cite[Prop. 1.6(a)]{Xu18} as ingredient (III) in Xu's proof, and  the results of Theorem \ref{thm_Xu_original} and Corollary \ref{cor_Xu_original} follow under our alternative hypotheses, thus verifying Theorem \ref{thm_F_power}.

We prove Proposition \ref{cor_count'} as follows.
Suppose that  $F$ is $(d,q)$-admissible, and write $F={G}^d\widetilde{F}$ in which $G, \widetilde{F} \in \kappa(x_1,\ldots,x_n)$ and $\widetilde{F}$ is $d$-th power-free, with $\deg \widetilde{F}<  \mathrm{char}\, \kappa$. 
Under the assumption that $F$ is $(d,q)$-admissible, by definition there is no linear change of variables $A \in \mathrm{GL}_n(\F_q)$ such that $\widetilde{F}(\xbf A) \in \F_q[x_2,\ldots, x_n]$. Hence by Lemma \ref{lemma_equiv_k} (which we may apply since $\deg \widetilde{F}<q$), the only value of $\mbf \in \ov{\kappa}^n$ such that $\widetilde{F}(\xbf) \con \widetilde{F}(\xbf + \mbf)$ is $\mbf = \zerobf$, so that 
\[ T_{\widetilde{F}}:=\{\mbf \in\ov{\kappa}^n\mid \widetilde{F}(\xbf)\equiv \widetilde{F}(\xbf+\mbf)\}\]
 is certainly finite.
Thus we may apply the following lemma to $\widetilde{F}$ (which we quote without repeating the proof):

\begin{lemma}[{Lemma 3.16 of \cite{Xu18}}]\label{count}
Fix $r, D \geq 1$. There exists $C_0 =C_0(r,D)$ such that the following holds. Let $\kappa$ be a finite field and $\kappa_0$ its prime field.
Let $H\in\kappa(x_1,\dots,x_n)$ be a $d$-th power-free rational function of degree at most $D$, and assume that
\[ T_H=\{\mbf \in\ov{\kappa}^n\mid H(\xbf)\equiv H(\xbf+\mbf)\}\]
  is finite. For any finite extension $k/\kappa$  and $\{a_i\}_{i=1}^{2r}\subset\Z$ such that $\gcd(d,a_i)=1$ for each $1 \leq i \leq 2r$, let
$P_H$ be the collection of tuples $(\mbf\up1,\dots,\mbf\up {2r})\in k^{2nr}$ such that the rational function $\prod_{i=1}^{2r} H(\xbf +\mbf \up i)^{a_i}$ is a perfect $d$-th power over $\ov{\kappa}$. 
Then 
\beq\label{P_bound}
\#P_H\le C_0(\#k)^{nr}(\#T_H)^{r}.
\eeq
\end{lemma}

Now note that for any such set of exponents $\{a_i\}_{i=1}^{2r}$, we have that $\prod_{i=1}^{2r} \widetilde{F}(\xbf+\mbf^{(i)})^{a_i}$ is a perfect $d$-th power if and only if 
\[ \prod_{i=1}^{2r} F(\xbf+\mbf^{(i)})^{a_i}  
= \prod_{i=1}^{2r} \widetilde{F}(\xbf+\mbf^{(i)})^{a_i} (\prod_{i=1}^{2r} G(\xbf+\mbf^{(i)})^{a_i})^d\]
 is. Thus if we define $P_{F}$, respectively $P_{\widetilde{F}}$, to be the collection of tuples $(\mbf\up1,\dots,\mbf\up {2r})\in k^{2nr}$ such that the rational function $\prod_{i=1}^{2r} F(\xbf +\mbf \up i)^{a_i}$ is a perfect $d$-th power over $\ov{\kappa}$ (or analogously for $\widetilde{F}$), we have by Lemma \ref{count} that 
 \[
 \# P_F  = \# P_{\widetilde{F}} \leq C_0 (\# k)^{nr}.
 \]
This concludes the verification of  Proposition \ref{cor_count'} and hence of Theorem \ref{thm_F_power}.

\section{Initiation of  the Burgess argument}\label{sec_Burgess_argument}
We now derive the first steps of the Burgess method,  generalizing  the approach of  \cite{GalMon10,HB13} from the one-variable case and \cite{Pie16} in the multi-variable case. The central complications that are distinctive to our new stratified setting are mainly addressed in Section \ref{sec_MR}.

We make the observation that given a character sum $S(F;\Nbf,\Hbf)$ with  data $F, \Nbf,\Hbf$, we may re-order the variables $x_1,\ldots, x_n$ so that the lengths $H_1,\ldots, H_n$ satisfy 
\beq\label{H_ordering}
H_1 \leq H_2 \leq \cdots \leq H_n.
\eeq
In particular, if $F$ is $(\Del,q)$-admissible, then it stays $(\Del,q)$-admissible after any re-ordering of the variables (Lemma \ref{lemma_perm}). 
We will assume (\ref{H_ordering}) from now on, and will prove the statement of Theorem \ref{thm_main_mult} with $H_1=H_{\min}$ and $H_n=H_{\max}$. 

 Fix a prime $p \ndiv q$ such that $p \leq H_{\min}$, and split each coordinate $x_i \in (N_i,N_i+H_i]$ into residue classes modulo $p$ by writing $\xbf = \abf q + p \mbf$ with $\abf, \mbf \in \Z^n$, where $0 \leq a_i <p$ and $m_i \in (N_i',N_i'+H_i']$, for which we define
 \begin{align*}
 N_i' &= (N_i-a_i q)/p, \\
 H_i' & = H_i/p
 \end{align*}
  for each $i=1, \ldots, n$. 
 Then 
\[ S(F;\Nbf,\Hbf) = \sum_{ \abf \modd{p}}\sum_{\mbf \in (\Nbf', \Nbf' + \Hbf']} \chi(F(\abf q + p \mbf)).\]
By the fact that $\chi$ has period $q$, the homogeneity of $F$ and the multiplicativity of $\chi$, 
\[ S(F;\Nbf,\Hbf) =  \chi(p^D) \sum_{ \abf \modd{p}} \sum_{\mbf \in (\Nbf', \Nbf' + \Hbf']} \chi(F(\mbf)),\] 
so that 
\[ |S(F;\Nbf,\Hbf) | \leq  \sum_{ \abf \modd{p}} | S(F;\Nbf',\Hbf')|.\]
We now average over a set $\Pcal$ of primes, $\Pcal = \{ P<p \leq 2P: p \ndiv q\}$, so that $| \Pcal| \gg P(\log P)^{-1}$. 
We will later choose $P$ so that 
\beq\label{P_small}
P \leq H_i, \qquad  1 \leq i \leq n.
\eeq
Then 
\beq\label{SFNH_D}
 |S(F;\Nbf,\Hbf)| \leq \frac{1}{| \Pcal|} \sum_{p \in \Pcal} \sum_{\abf \modd{p}} |S(F;\Nbf',\Hbf')|.
 \eeq

Here we recall that $\Nbf'$ and $\Hbf'$ depend on $\abf$ and $p$.
We now average over the starting points $\Nbf'$ and the lengths $\Hbf'$ in order to make them independent of $\abf, p$.

\begin{lemma}\label{lemma_blur_D}
 For any $\Kbf \leq \Lbf$,
\[  |S(F;\Mbf,\Kbf)| \leq 2^{2n} \| \Lbf \|^{-1} \sum_{\mbf \in (\Mbf - \Lbf, \Mbf]} \max_{\kbf \leq 2\Lbf} |S(F;\mbf,\kbf)|.\]
\end{lemma}\label{lemma_blur}
This   follows by inclusion-exclusion, which shows that for any $\mbf$ with $m_i \leq M_i$ for all $i = 1,\ldots, n$, 
\[ S(F;\Mbf,\Kbf) = \sum_{\bstack{\delbf = (\del_1,\ldots, \del_n)}{\del_i \in \{0,1\}}} (-1)^{\sig(\delbf)} S(F;\mbf,\Mbf-\mbf + (\onebf -\delbf )\cdot \Kbf),\]
where $\sig(\delbf)= \sum_i \del_i$.
For any $K_i \leq L_i$, and any $m_i$ with $M_i-L_i < m_i \leq M_i$, we have 
\[ 0 \leq M_i - m_i + (1-\del_i)K_i \leq 2L_i,\]
for either choice of $\del_i \in \{0,1\}$.
Hence for any $\mbf$ with $\Mbf - \Lbf \leq \mbf \leq \Mbf$,
\[ S(F;\Mbf,\Kbf) \leq 2^n \max_{\kbf \leq 2\Lbf} |S(F;\mbf,\kbf)|.\]
There are at least $L_i/2$ integers with $M_i - L_i < m_i \leq M_i$, so that there are at least $2^{-n}\| \Lbf \|$ values $\mbf$ in the range $\Mbf - \Lbf <\mbf \leq \Mbf$, and the lemma now follows by averaging over these values.

We apply this lemma to (\ref{SFNH_D}) with $L_i=H_i/P \geq 1$. We obtain
\beq\label{SFNH_D2}
 |S(F;\Nbf,\Hbf)| \leq |\Pcal|^{-1} \sum_{p \in \Pcal} \sum_{\abf \modd{p}} 2^{2n} \| \Hbf \|^{-1} P^{n} \sum_{\mbf \in (\Nbf' - \Hbf/P,\Nbf']} \max_{\kbf \leq 2\Hbf/P} |S(F;\mbf,\kbf)|.
 \eeq
 After rearranging,
 \beq\label{SFNH_D3}
  |S(F;\Nbf,\Hbf)| \leq|\Pcal|^{-1} 2^{2n} \| \Hbf \|^{-1} P^{n} \sum_{\mbf} \Acal(\mbf) \max_{\kbf \leq 2 \Hbf /P} |S(F;\mbf,\kbf)|,
  \eeq
 where we have defined
 \begin{multline*}
  \Acal(\mbf) = \#\{ \abf,p : p \in \Pcal, \abf = (a_1,\ldots, a_n), 0 \leq a_i < p: 
  	\frac{N_i-a_iq}{p} - \frac{H_i}{P} < m_i \leq \frac{N_i-a_iq}{p}, i=1,\ldots, n \}.
  \end{multline*}
  We record the following facts about $\Acal(\mbf)$, whose proof we defer to Section \ref{sec_tech}.
\begin{lemma}\label{lemma_A_D}
The quantity $\Acal(\mbf)$ vanishes unless $\mbf$ satisfies $|m_i| \leq 2q$ for each $i$. Moreover, if 
\beq\label{P_big}
H_iP < q \qquad 1 \leq i \leq n,
\eeq
 then
\[ \sum_{\mbf} \Acal(\mbf) \ll P \| \Hbf \|, \qquad \sum_\mbf \Acal(\mbf)^2 \ll  P \| \Hbf \|. \]
\end{lemma}
Applying H\"{o}lder's inequality twice to (\ref{SFNH_D3}), we obtain 
\[
 |S(F;\Nbf,\Hbf)| \ll | \Pcal|^{-1} \| \Hbf \|^{-1} P^{n} \left(\sum_{\mbf} \Acal(\mbf)\right)^{1-1/r} \left(\sum_{\mbf} \Acal(\mbf)^2 \right)^{1/(2r)} \\
 \left(\sum_{\bstack{\mbf}{ |m_i| \leq 2q}} \max_{\kbf \leq 2\Hbf /P} |S(F;\mbf,\kbf)|^{2r}\right)^{1/(2r)} .
 \]
 Thus applying the results of Lemma \ref{lemma_A_D} shows that 
 \beq\label{SFNH_D4}
  |S(F;\Nbf,\Hbf)| \ll  (\log P) P^{n-1/2r} \| \Hbf \|^{-1/2r} \left(\sum_{\bstack{\mbf}{ |m_i| \leq 2q}} \max_{\kbf \leq 2\Hbf/P} |S(F;\mbf,\kbf)|^{2r}\right)^{1/(2r)} .
 \eeq
 It is sufficient to look at the internal sum over $\mbf$ modulo $q$; in fact obtaining this complete character sum is the main accomplishment of the manipulations up to this point.
 
 We ignore for the moment the maximum over $\kbf \leq 2\Hbf/P$ and focus first on estimating the non-maximal moment.
  We re-write $S(F;\mbf,\kbf)$ as
 \begin{eqnarray*}
  S(F;\mbf,\kbf)
 	= \sum_{\xbf \in (\mbf , \mbf + \kbf]} \chi(F(\xbf))  = \sum_{\xbf \in (\zerobf,\kbf]} \chi(F(\mbf +\xbf)), 
	\end{eqnarray*}
so that upon expansion,
\beq\label{SG}
 \sum_{\mbf \modd{q}} |S(F;\mbf,\kbf)|^{2r}
 =\sum_{\mbf\modd{q}} \left|  \sum_{\xbf \in (\zerobf,\kbf]} \chi(F(\mbf +\xbf)) \right|^{2r}
   \leq \sum_{\bstack{\xbf^{(1)},\ldots,\xbf^{(2r)} }
   {\xbf\up i\in(\zerobf,\kbf]}} \left| \sum_{\mbf \modd{q}} \chi (F_{\{\xbf\}}(\mbf)) \right|,
 \eeq
where $F_{\{\xbf\}}$ is defined in terms of the original form $F$ and the collection $\{\xbf^{(1)},\ldots,\xbf^{(2r)}\}$ by (\ref{dfn_Gpoly}).

\section{Stratification and a Menchov--Rademacher argument}\label{sec_MR}
\subsection{Application of the stratification for character sums}
We now come to a critical novel step, which is to estimate how often we obtain a certain quality of upper bound for the complete character sum
\[ \sum_{\mbf \modd{q}} \chi (F_{\{\xbf\}}(\mbf)).\]
For this, we call upon the stratification of complete character sums stated in Theorem \ref{thm_Xu}.
Let us suppose that $k_1 \leq \ldots \leq k_n \leq q$.
 For each $1 \leq j \leq n$, define 
\[ Y_j:= \left\{ \{\xbf\} 
\in (\zerobf,\kbf]^{2r} : \left|  \sum_{\mbf \modd{q}}
\chi(F_{\{\xbf\}}(\mbf))
\right|  > C q^{(n+j-1)/2} \right\},\]
so that $(\zerobf,\kbf]^{2r} =:Y_0\supset Y_1\supset Y_2\supset\dots\supset Y_n\supset Y_{n+1}:=\varnothing$, and $\#Y_j\leq C'' \norm{\kbf}^{2r}B_{n,r,j}(\kbf)^{-1}$ by Theorem \ref{thm_Xu}. Upon employing the decomposition $(\zerobf,\kbf]^{2r}=\coprod_{j=0}^n Y_j\setminus Y_{j+1}$ in (\ref{SG}) 
 we have
\begin{align}
\sum_{\mbf\modd{q}}\abs{S(F;\mbf,\kbf)}^{2r}
& \le  \sum_{j=0}^n \sum_{\{\xbf\}\in Y_j\setminus Y_{j+1}} \abs*{\sum_{\mbf\modd{q}}\chi(F_{\{\xbf\}}(\mbf))} \label{Y_original}\\
& \le \sum_{j=0}^n (\#Y_j) Cq^{(n+(j+1)-1)/2} \nonumber \\
& \leq C \cdot C'' \|\kbf\|^{2r}\sum_{j=0}^n q^{(n+j)/2}B_{n,r,j}(\kbf)^{-1}. \nonumber
\end{align}

To summarize, we have proved:

\begin{lemma}\label{moment_bound}
Fix  $n,r,D,\Del \geq 1$, a prime $q$, and a multiplicative Dirichlet character $\chi$ of order $\Del$ modulo $q$. Then for all forms $F \in \Z[X_1,\ldots,X_n]$ of degree $D$ that are $(\Del,q)$-admissible, the following holds. For every  $\kbf = (k_1,\ldots, k_n)$ with  $1 \leq k_1\le k_2\le\dots\le k_n \leq q$, then
\beq\label{MR_weak1}
\sum_{{\bf m}\pmod{q}} \abs{S(F;{\bf m,k})}^{2r}\ll_{n,r,D,\Del} \norm{{\bf k}}^{2r} \sum_{j=0}^n q^{(n+j)/2}B_{n,r,j}(\kbf)^{-1},
\eeq 
with $B_{n,r,j}(\kbf)$ as defined in (\ref{Bnrj_dfn}), and with the implicit constant dependent on $n,r,D,\Del$ but independent of $q,\chi,F,\kbf$.
\end{lemma}
However recall that the actual quantity we must bound in (\ref{SFNH_D4}) is a moment of $|S(F;\mbf,\kbf)|$ that includes a maximum over $\kbf \leq 2\Hbf/P$. To do so, we will employ a Menchov-Rademacher argument.

\subsection{A Menchov-Rademacher argument with permutations}
The  Menchov-Rademacher argument \cite{Men23,Rad22} may be employed in a wide variety of circumstances; in general it allows one to replace a supremum of a function $|f(u_j)|$ over an index set of size $U$ by a sum of differences $|f(u_j) - f(u_{j-1})|$ over an index set of size $O(\log U)$. In our present setting, these differences are differences of partial sums, which are themselves partial sums of the same kind, so that the Menchov-Rademacher device is a useful tool. 

However   a typical Menchov-Rademacher argument would not immediately apply in our case, since we cannot save a power of $\| \kbf\|$ but typically only a power of $k_{\min}$, the shortest side of the box. We see this phenomenon is already present in Lemma \ref{moment_bound}, since we have assumed an ordering $k_1 \leq k_2 \leq \cdots \leq k_n$ for the side-lengths of the box. Even if we assume in the beginning that we have such an ordering, certain internal steps in the Menchov-Rademacher argument do not preserve such an ordering, and thus arranging the argument so that we may apply Lemma \ref{moment_bound} will require delicate considerations of permutations of the variables.

Our main result in this section is the following: 
\begin{prop}\label{prop_MR}
If $1 \leq K_1\le K_2\le\dots\le K_n \leq q$, then
\beq\label{MR_weak2}
\sum_{\substack{{\bf m}\\\abs{m_i}\le2q}}\max_{{\bf k}\le {\bf K}} \abs{S(F;{\bf m,k})}^{2r}\ll_{n,r,D,\Del} \norm{{\bf K}}^{2r} (\log K_n)^{2nr} \sum_{j=0}^n q^{(n+j)/2} \widetilde{B}_{n,r,j}(\Kbf)^{-1}
\eeq 
with the implicit constant dependent on $n,r,D,\Del$ but independent of $q,\chi,\Kbf$, and with
 
\beq\label{tilde_B_dfn}
\widetilde{B}_{n,r,j}(\Kbf) = \begin{cases}
	1 & j=0\\
	K_1^{\theta_j} & j=1,\ldots, n-1\\
	(K_1 \cdots K_{n/2})^{(2r-1)} & j=n, \text{$n$ even} \\
	(K_1 \cdots K_{(n-1)/2})^{(2r-1)} (K_{(n+1)/2})^{r} & j=n, \text{$n$ odd}
	\end{cases} 
\eeq
\end{prop}
 
Note that $\widetilde{B}_{n,r,n}(\Kbf)^{-1}$ loses one power in decay compared to $B_{n,r,n}(\Kbf)^{-1}$. In contrast, in the one-dimensional case, no decay is lost when passing from the non-maximal estimate to the maximal estimate; this minor loss will not affect our final outcome, and we explain why it arises below.

\subsection{A dyadic decomposition}
It suffices to prove Proposition \ref{prop_MR} for $\Kbf = (K_1,\ldots,K_n)$ where each $K_i$ is a power of $2$, say $K_i=2^{t_i}$, in which case we naturally have $t_1 \leq t_2 \leq \cdots \leq t_n$, under the hypotheses of the proposition. We fix $\mbf$ and suppose that $S(F;\mbf,\kbf)$ assumes its maximum for $\kbf \leq \Kbf$ at $\kbf=(k_1,\ldots,k_n)$, and for each $i$ we decompose 
\[ k_i = \sum_{\del_i \in \Dcal_i} 2^{t_i-\del_i},\]
 where $\Dcal_i$ is a set of distinct non-negative integers $\del_i \leq t_i$.
 Let $\delbf=(\del_1,\ldots,\del_n)$ denote a tuple with $\del_i \in \Dcal_i$ for each $i$. Let $\tbf=(t_1,\ldots,t_n)$ and define  the notation $2^{\tbf-\delbf} = (2^{t_1-\del_1},\ldots, 2^{t_n-\del_n})$. Furthermore let $\Vbf_{\mbf,\delbf}$ be an $n$-tuple defined such that the  $i$-th entry is $(\Vbf_{\mbf,\delbf})_i  = 2^{t_i-\del_i}v_{\mbf,\delbf,i}$, where
\[ v_{\mbf,\delbf,i}  = \sum_{e_i \in \Dcal_i, e_i < \del_i} 2^{\del_i - e_i} < 2^{\del_i}.\]
Then we may express
\[ S(F;\mbf,\kbf)
	 = \sum_{\delbf, \del_i \in \Dcal_i} S(F; \mbf+ \Vbf_{\mbf,\delbf}, 2^{\tbf-\delbf}).\]
After taking absolute values, we can obtain an upper bound by increasing the sum to run over all tuples $\delbf$ with $\del_i \leq t_i$. We get
\[ |S(F;\mbf,\kbf)|
	  \leq  \sum_{\bstack{\delbf}{0 \leq \del_i \leq t_i}} |S(F; \mbf+ \Vbf_{\mbf,\delbf}, 2^{\tbf-\delbf})|.
\]
Then by H\"{o}lder's inequality, 
\[ |S(F;\mbf,\kbf)|^{2r}  \leq  (\prod_{i=1}^n(t_i+1))^{2r-1}\sum_{\bstack{\delbf}{0 \leq \del_i \leq t_i}} \sum_{\bstack{\vbf}{0 \leq v_i < 2^{\del_i}}} 
 	|S(F; \mbf+ 2^{\tbf-\delbf} \cdot \vbf, 2^{\tbf-\delbf})|^{2r}, \]
where we have possibly enlarged the right hand side by summing over all possible values for $v_{\mbf,\delbf,i}$ up to $2^{\del_i}$. 
Recalling that for each $\mbf$ we chose $\kbf$ above to be the length at which the maximum is attained, we then have
\begin{align*}
 \sum_{|m| \leq 2q} \max_{\kbf  \leq 2\Kbf} |S(F;\mbf,\kbf)|^{2r} 
&	\leq   T^{2r-1}\sum_{\bstack{\delbf}{0 \leq \del_i \leq t_i}} \sum_{\bstack{\vbf}{0 \leq v_i < 2^{\del_i}}} \sum_{\bstack{\mbf}{|m_i| \leq 2q}} |S(F; \mbf+ 2^{\tbf-\delbf} \cdot \vbf, 2^{\tbf - \delbf})|^{2r}\\
&	\ll  T^{2r-1} \sum_{\bstack{\delbf}{0 \leq \del_i \leq t_i}} \norm{2^{\delbf}} \sum_{\mbf\modd{q}} \abs{S(F;\mbf,2^{\tbf-\delbf})}^{2r},
\end{align*}
where $T:=\prod_{i=1}^n(t_i+1)$.
 Now we perform the key step that accommodates the  fact that we only make a savings in the smallest direction of the box.
  
\subsection{Application of the non-maximal upper bound}
We would like to apply Lemma \ref{moment_bound} to the innermost sums over $\mbf$. Fix $\delbf$ (with $0 \leq \del_i \leq t_i$ for each $i$).
Notice that we may not have $t_1-\delta_1\le t_2-\delta_2\le \dots\le t_n-\delta_n$ even though $t_1\le t_2\le\dots\le t_n$, so we may need to reorder $\tbf-\delbf$ before applying   Lemma \ref{moment_bound}.
Let $\sigma$ be a permutation of $\{1,2,\dots,n\}$ (depending on $\delbf$) such that 
\beq\label{t_d_growing}
t_{\sigma(1)}-\delta_{\sigma(1)}\le t_{\sigma(2)}-\delta_{\sigma(2)}\le\dots\le t_{\sigma(n)}-\delta_{\sigma(n)}.
\eeq
Given an $n$-tuple $\xbf$, let $\xbf_\sig = (x_{\sig(1)},x_{\sig(2)},\ldots,x_{\sig(n)})$, and   $2^{\tbf_\sig - \delbf_\sig} = (2^{t_{\sig(1)} - \del_{\sig(1)}}, \ldots, 2^{t_{\sig(n)} - \del_{\sig(n)}})$. Recall the discussion on permutations of variables in Section \ref{sec_perm}; for any permutation $\pi$ of $\{1,\ldots, n\}$ define the form $F_\pi(\Xbf)$  by setting
$F_\pi(X_1,\ldots, X_n) = F(X_{\pi(1)},\ldots, X_{\pi(n)}).$ 

Letting $\sig^{-1}$ be the permutation inverse of $\sig$,  then
\[ S(F; \mbf,2^{\tbf-\delbf}) = \sum_{\xbf \in (\mbf,\mbf+2^{\tbf - \delbf})} \chi(F(\xbf)) = 
 \sum_{\xbf_\sig \in (\mbf_\sig, \mbf_\sig + 2^{\tbf_\sig - \delbf_\sig})} \chi( F_{\sig^{-1}}(\xbf_\sig))
 	 = S(F_{\sig^{-1}};\mbf_\sig,2^{\tbf_\sig - \delbf_\sig}). \]
Note that $\mbf_\sig$ ranges over all $n$-tuples with coordinates modulo $q$ as $\mbf$ does, and that the last argument in $S(F_{\sig^{-1}};\mbf_\sig,2^{\tbf_\sig - \delbf_\sig})$ satisfies the requirement  (\ref{t_d_growing}) so that we may apply Lemma \ref{moment_bound} with $n,r,D,\Del,q,\chi$ as before but now to the form $F_{\sig^{-1}}$; here we use the uniformity of Lemma \ref{moment_bound} with respect to the form. (We recall from Lemma \ref{lemma_perm} that $F_{\sig^{-1}}$ is $(\Del,q)$-admissible if and only if $F$ is.) 
We may conclude (using the fact that   $ \|2^{\tbf_\sig - \delbf_\sig}\| = \|2^{\tbf - \delbf}\| =\|2^{\tbf}\| \cdot \|2^{\delbf}\|^{-1}$) that 

\begin{align}
 \sum_{|m| \leq 2q} \max_{\kbf  \leq \Kbf} |S(F;\mbf,\kbf)|^{2r} 
&	\ll_{n,r,D,\Del} T^{2r-1}\sum_{\bstack{\delbf}{0 \leq \del_i \leq t_i}} \norm{2^{\delbf}} \norm{2^{\tbf_\sig-\delbf_\sig}}^{2r} \sum_{j=0}^n q^{(n+j)/2}B_{n,r,j}(2^{\tbf_\sig-\delbf_\sig})^{-1} \nonumber \\
 & =  T^{2r-1}\sum_{\bstack{\delbf}{0 \leq \del_i \leq t_i}} \norm{2^\tbf}^{2r} \norm{2^{\delbf}}^{-(2r-1)}\sum_{j=0}^n q^{(n+j)/2}B_{n,r,j}(2^{\tbf_\sig-\delbf_\sig})^{-1}\nonumber  \\
 & =  T^{2r-1}\norm{2^\tbf}^{2r}\sum_{j=0}^n q^{(n+j)/2}  \sum_{\bstack{\delbf}{0 \leq \del_i \leq t_i}} \norm{2^{\delbf}}^{-(2r-1)}B_{n,r,j}(2^{\tbf_\sig-\delbf_\sig})^{-1} \nonumber \\
 & \le  T^{2r}\norm{\Kbf}^{2r}\sum_{j=0}^n q^{(n+j)/2}  \max_{\bstack{\delbf}{0\le \del_i\le t_i}}  \left\{ \norm{2^{\delbf}}^{-(2r-1)}B_{n,r,j}(2^{\tbf_\sig-\delbf_\sig})^{-1} \right\}.\label{S_K_MR_upper}
  \end{align}
  Note that in the case of dimension $n=1$ (and $\theta_1=r$), 
  the sum over $j$ is comprised of the two terms
  \[ q^{n/2} 2^{-\del(2r-1)} + q^{n}2^{-\del(2r-1)}(2^{t-\del})^{-r} \leq q^{n/2} + q^{n} 2^{-rt} = q^{n/2} + q^{n} K^{-r},\]
  familiar from the classical 1-dimensional Burgess argument.
  
Now in general for $n \geq 2$ we must re-interpret $B_{n,r,j}(2^{\tbf_\sig-\delbf_\sig})^{-1}$ in terms of the coordinates of $\Kbf = 2^\tbf$, in which we recall that $t_1 \leq t_2 \leq \cdots \leq t_n$; this argument is more complicated, and in particular for $j=n$ we will get a positive power of $2^{\del}$ we cannot ignore as  in the case of $n=1$ (e.g. compare to the top line of \cite[page  204]{HB13}). We summarize the necessary result:

  \begin{lemma}\label{lemma_MR_savings}
For $\Kbf = 2^{\tbf}$ with $t_1 \leq \cdots \leq t_n$, for each $\delbf \leq \tbf$, let $\sig$ be a permutation of indices such that $t_{\sig(1)} - \del_{\sig(1)} \leq \cdots \leq t_{\sig(n)} - \del_{\sig(n)}$. Then
\[  \norm{2^{\delbf}}^{-(2r-1)}B_{n,r,j}(2^{\tbf_\sig-\delbf_\sig})^{-1}  \leq 
	\begin{cases}
	1 & j=0 \\
	K_1^{-\theta_j} & j=1,\ldots, n-1\\
	(K_1 \cdots K_{n/2})^{-(2r-1)} & j=n, \text{$n$ even} \\
	(K_1 \cdots K_{(n-1)/2})^{-(2r-1)}K_{(n+1)/2}^{-r} & j=n, \text{$n$ odd}.\\
	\end{cases} \]
  \end{lemma}
  Once we apply this lemma to (\ref{S_K_MR_upper}), upon noting that $T^{2r} \leq 2^{2nr} (\log_2 K_n)^{2nr}$
we have proved Proposition \ref{prop_MR}. 

\subsection{Proof of Lemma \ref{lemma_MR_savings}: rearrangement}
We recall the definition of $B_{n,r,j}(\kbf)$ in (\ref{Bnrj_dfn}) so that for a tuple $\kbf = (k_1,\ldots,k_n)$ with $k_1 \leq k_2 \leq \cdots \leq k_n$,
   $B_{n,r,j}(\kbf)=1$ for $j=0$,
  $B_{n,r,j}(\kbf) = k_1^{\theta_j}$ for $j=1,\ldots, n-1$, and for $j=n$ we have 
  $B_{n,r,j}(\kbf) = (k_1 \cdots k_{n/2})^{2r}$ if $n$ is even and 
    $B_{n,r,j}(\kbf) = (k_1 \cdots k_{(n-1)/2})^{2r} k_{(n+1)/2}^{r}$ if $n$ is odd.
    
    We may quickly dispatch the case $j=0$, in which case
    \[  \norm{2^{\delbf}}^{-(2r-1)}B_{n,r,j}(2^{\tbf_\sig-\delbf_\sig})^{-1}  
    	 =  \norm{2^{\delbf}}^{-(2r-1)}  \leq 1.\]
    For the remaining cases of $j \geq 1$, it is helpful to invert, and take the logarithm, and prove for fixed $\delbf$ and fixed $j$ a lower bound for the quantity
\[  \log_2\parens*{\norm{2^{\delbf}}^{(2r-1)}B_{n,r,j}(2^{\tbf_\sig-\delbf_\sig})} .\]
First we consider the case of $1 \leq j \leq n-1$; for each of these $j$ (using the fact that $\sum_i \del_{\sig(i)} = \sum_i \del_i$),
\begin{align*}
  \log_2\parens*{\norm{2^{\delbf}}^{(2r-1)}B_{n,r,j}(2^{\tbf_\sig-\delbf_\sig})} 
& = (2r-1)\sum_{i=1}^n\delta_{\sigma(i)} + \theta_j(t_{\sigma(1)}-\delta_{\sigma(1)}) \\
& =  \theta_j t_{\sigma(1)} + (2r-1 - \theta_j) \del_{\sig(1)} + (2r-1) \sum_{i=2}^n \del_{\sig(i)} \geq \theta_j t_{\sigma(1)}.
\end{align*}
Here we used that for $1 \leq j \leq n-1$ we have $\theta_j \leq r-1$, and moreover $\del_{\sig(i)} \geq 0$ for all $1 \leq i \leq n$. Thus we have 
 \[ \norm{2^{\delbf}}^{-(2r-1)}B_{n,r,j}(2^{\tbf_\sig-\delbf_\sig})^{-1} \leq 2^{-\theta_j t_{\sig(1)}} \leq 2^{-\theta_j t_1} = K_1^{-\theta_j},
 \]
 upon recalling that $t_1 \leq t_2 \leq \cdots \leq t_n$. 
 
 Now we turn to the more complicated case of $j=n$. First we assume  that $n$ is even. Now we have
\[
  \log_2\parens*{\norm{2^{\delbf}}^{(2r-1)}B_{n,r,j}(2^{\tbf_\sig-\delbf_\sig})} 
 = (2r-1)\sum_{i=1}^n\delta_{\sigma(i)} +  2r \sum_{i=1}^{n/2}(t_{\sigma(i)}-\delta_{\sigma(i)}) .\]
It is convenient to set temporarily for each $i=1,\ldots,n$ the parameter $\Theta_{n,i}  = 2r$ if $1 \leq i \leq n/2$ and $\Theta_{n,i} = 0$ if $n/2 < i \leq n$. Then upon recalling that  each  $\del_i \leq t_i$, we have
 \begin{align}
  \log_2\parens*{\norm{2^{\delbf}}^{(2r-1)}B_{n,r,j}(2^{\tbf_\sig-\delbf_\sig})} 
  &= (2r-1)\sum_{i=1}^n\delta_{\sigma(i)} + \sum_{i=1}^{n}\Theta_{n,i}(t_{\sigma(i)}-\delta_{\sigma(i)}) \nonumber \\
& = \sum_{i=1}^{n}\Theta_{n,i}t_{\sigma(i)} + \sum_{i=1}^n(2r-1-\Theta_{n,i})\delta_{\sigma(i)} \nonumber \\
& \geq \sum_{i=1}^{n}\Theta_{n,i}t_{\sigma(i)} + \sum_{\bstack{i=1}{2r - 1 - \Theta_{n,i}<0}}^n(2r-1-\Theta_{n,i})t_{\sigma(i)}+
	\sum_{\bstack{i=1}{2r - 1 - \Theta_{n,i} \geq 0}}^n(2r-1-\Theta_{n,i})0 \nonumber  \\
	& = \sum_{\bstack{i=1}{2r - 1 - \Theta_{n,i}<0}}^n(2r-1)t_{\sigma(i)}+
	\sum_{\bstack{i=1}{2r - 1 - \Theta_{n,i} \geq 0}}^n\Theta_{n,i} t_{\sig(i)}. \label{argument_t}
\end{align}
(Here in the inequality, equality can occur for those $\delbf$ such that $\del_i = t_i$ for all $1 \leq i \leq n/2$ and $\del_i = 0$ for $n/2< i \leq n$.
The inequality is where we see that if $\Theta_{n,i}=2r$, we must replace $\del_{\sig(i)}$ by $t_{\sig(i)}$ rather than by $0$; this is why in the final statement of the inequality we lose slightly in the maximal moment, compared to the non-maximal moment. This effect is not possible in dimension $n=1$.)
 Now by the definition of $\Theta_{n,i}$, the first sum is over $1 \leq i \leq n/2$; the second sum is over $n/2 < i \leq n$, in which range $\Theta_{n,i}=0$ so that the second sum is vacuous. Now we use the following simple observation.
 \begin{lemma}[Rearrangement inequality]\label{lemma_rearrange}
 Let $t_1 \leq t_2 \leq \cdots \leq t_n$ be a fixed non-decreasing sequence of real numbers and $a_1 \geq \cdots \geq a_n$ a fixed non-increasing sequence of real numbers. Then for any permutation $\sigma$ on $\{1,\ldots, n\}$, and for any $1 \leq M \leq n$,
  \beq\label{sum_rearrangement_with_coeffs}
    \sum_{i=1}^{M}a_i t_{i} \leq  \sum_{i=1}^{M}a_i t_{\sigma(i)}.
  \eeq
  \end{lemma}
 
 This is a variant of a standard rearrangement inequality; for completeness we give a brief proof in \S \ref{sec_tech}.
 Applying this observation in (\ref{argument_t}) with $M=n/2$, we have shown that 
\[  \log_2\parens*{\norm{2^{\delbf}}^{(2r-1)}B_{n,r,j}(2^{\tbf_\sig-\delbf_\sig})}  \geq (2r-1) \sum_{i=1}^{n/2} t_i,\]
so that in the case of $j=n$ even,
 \[ \norm{2^{\delbf}}^{-(2r-1)}B_{n,r,j}(2^{\tbf_\sig-\delbf_\sig})^{-1} \leq (2^{t_1}\cdots2^{t_{n/2}})^{-(2r-1)} = (K_1 \cdots K_{n/2})^{-(2r-1)}.
 \]
 
 The argument is similar for $j=n$ with $n$ odd, and we only specify the necessary changes, starting with 
\[
  \log_2\parens*{\norm{2^{\delbf}}^{(2r-1)}B_{n,r,j}(2^{\tbf_\sig-\delbf_\sig})} 
 = (2r-1)\sum_{i=1}^n\delta_{\sigma(i)} +  2r \sum_{i=1}^{(n-1)/2}(t_{\sigma(i)}-\delta_{\sigma(i)}) +r(t_{\sig(\frac{n+1}{2})} - \del_{\sig(\frac{n+1}{2})}) .\]
It is convenient to set temporarily for each $i=1,\ldots,n$ the parameter $\Theta_{n,i}  = 2r$ if $1 \leq i \leq (n-1)/2$, $\Theta_{n,\frac{n+1}{2}} = r $ and $\Theta_{n,i} = 0$ if $(n+1)/2 < i \leq n$. With this notation, the argument then proceeds as before, until we reach the statement of (\ref{argument_t}), which now holds with this new definition of $\Theta_{n,i}$. Now the first sum on the right-hand side of (\ref{argument_t}) is over $1 \leq i \leq (n-1)/2$, while the second sum on the right-hand side is over $i \geq (n+1)/2$, and has its only non-zero contribution coming from $i = (n+1)/2$. We may conclude that 
\[ 
  \log_2\parens*{\norm{2^{\delbf}}^{(2r-1)}B_{n,r,j}(2^{\tbf_\sig-\delbf_\sig})} 
  \geq 
  \sum_{i=1}^{(n-1)/2} (2r-1)t_{\sigma(i)}+
	r t_{\sig((n+1)/2)}.
  \]
  We now apply  (\ref{sum_rearrangement_with_coeffs}) from Lemma \ref{lemma_rearrange} to conclude that 
  \[ 
  \log_2\parens*{\norm{2^{\delbf}}^{(2r-1)}B_{n,r,j}(2^{\tbf_\sig-\delbf_\sig})} 
  \geq 
  \sum_{i=1}^{(n-1)/2} (2r-1)t_{i}+
	r t_{(n+1)/2}, 
  \]
  or equivalently,
 \[ \norm{2^{\delbf}}^{-(2r-1)}B_{n,r,j}(2^{\tbf_\sig-\delbf_\sig})^{-1} \leq (2^{t_1}\cdots2^{t_{(n-1)/2}})^{-(2r-1)} (2^{t_{(n+1)/2}})^{-r}= (K_1 \cdots K_{(n-1)/2})^{-(2r-1)} K_{(n+1)/2}^{-r}.
 \]
 This completes the proof of Lemma \ref{lemma_MR_savings}.

\section{Conclusion of the Burgess argument}\label{sec_Burgess_conclude}

We now apply Proposition \ref{prop_MR} to (\ref{SFNH_D4}) with $\Kbf = 2\Hbf/P$, recalling that we are working under the assumption (\ref{H_ordering}) that $H_1\leq H_2 \leq \cdots \leq H_n$. (Also recall that $K_n \leq H_n < q$.) We conclude that 
\begin{eqnarray} 
 |S(F; \Nbf,\Hbf)|^{2r} &\ll_{n,r,\Del,D} & (\log q)^{2r(n+1)} P^{2nr-1} \| \Hbf \|^{-1} \parens*{ \frac{\norm{\Hbf}}{P^n}}^{2r} \sum_{j=0}^n q^{(n+j)/2}\widetilde{B}_{n,r,j}(\Hbf/P)^{-1}  \nonumber\\
& \le & (\log q)^{2r(n+1)}  \norm{\Hbf}^{2r-1}P^{-1}q^{n/2} \sum_{j=0}^n q^{j/2}\widetilde{B}_{n,r,j}(\Hbf/P)^{-1},\label{concluding}
\end{eqnarray}
in which we recall the definition of $\widetilde{B}_{n,r,j}(\cdot)$ from (\ref{tilde_B_dfn}).

At this stage of the Burgess argument in the one-dimensional setting $n=1$, one knows that $\theta_0=0$ and $\theta_1=r$, so that 
the sum over $j \in \{0,1\}$ contributes $(1+ q^{1/2} (H_1/P)^{-r})$.
To balance this, we would then choose $P$ to be an integer with 
\beq\label{P_choice_1}
 (1/2) H_1q^{-1/(2r)} \leq P \leq H_1 q^{-1/(2r)},
 \eeq
 which will appear familiar to experts.
Thus when $n=1$, we recover 
\[  |S(F; \Nbf,\Hbf)| \ll  \| \Hbf \|^{1 - 1/r} q^{ \frac{r + 1}{4r^2}}\log q ,\]
which agrees with Burgess's statement (\ref{Burgess_1}).
Now for $n \geq 2$, we observe:

\begin{lemma}\label{lemma_bound_sum}
For $1 \leq K_1 \leq K_2 \leq \cdots \leq K_n$,
\[ \sum_{j=0}^n q^{j/2} \widetilde{B}_{n,r,j}(\Kbf)^{-1} \ll_{n,r} 1
\]
precisely when 
\beq\label{K_relation}
q^{1/2}K_1^{-\theta_1} \leq 1.
\eeq
Under this assumption, the sum over $j$ is dominated by the terms with $j=0,1$.
\end{lemma}
In particular, this lemma (whose proof we defer to \S \ref{sec_tech}) shows that the sum is $\gg 1$ if the relation (\ref{K_relation}) does not hold; hence it is advantageous to assume (\ref{K_relation}). Under this assumption, we can dominate the sum by the terms with $j=0,1$ and hence we conclude from (\ref{concluding}) that
\[ |S(F; \Nbf,\Hbf)|^{2r} \ll_{n,r,\Del,D} (\log q)^{2r(n+1)}  \norm{\Hbf}^{2r-1}P^{-1}q^{n/2}  ( 1+ q^{1/2}(H_1/P)^{-\theta_1}).\]
To balance the last two terms within parentheses, we choose $P$ to be an integer with 
\beq\label{P_choice_n}
  \frac{H_1}{2q^{1/(2\theta_1)}} \leq P \leq \frac{H_1}{q^{1/(2\theta_1)}} ,
  \eeq
where we recall for the reader's convenience that $\theta_1= \floor*{\frac{r-1}{n-1}}$ if $n\ge2$.
 We recall that earlier in (\ref{P_small}) and (\ref{P_big}) we had the requirements that $P \leq H_i$ for all $i$ and  $H_iP < q$ for all $i$. The first is clearly true; the second we may verify as long as  we assume $H_{n}H_{1} < q^{1+1/(2\theta_1)}$, as we do in our theorem statement.

With this choice for $P$, we have
\[  |S(F; \Nbf,\Hbf)|^{2r} \ll_{n, r, \Del,D} \norm{\Hbf}^{2r-1} H_1^{-1} q^{n/2+1/(2\theta_1)}(\log q)^{2r(n+1)} , \]
and hence we conclude that
 \[ |S(F; \Nbf,\Hbf)| \ll_{n, r, \Del,D}  \norm{\Hbf}^{1-1/(2r)}H_1^{-1/(2r)}q^{\frac{n\theta_1+1}{4r\theta_1}}(\log q)^{n+1} .\]
 This proves Theorem \ref{thm_main_mult}, upon recalling that we have reduced to the setting in which $H_{\min} = H_1$, $H_{\max} = H_n$, and we have set $\Theta=\theta_1 = \floor*{\frac{r-1}{n-1}}$ if $n\ge2$.  

\subsection{The role of the stratification}\label{sec_strat_1}
It is useful to remark here on the crucial role that the stratification has played. Suppose that instead of Theorem \ref{thm_Xu_original} we only gained information about $X_1$, without any further stratification into $X_2, \ldots, X_n$. Then for those $(\xbf^{(1)}, \ldots, \xbf^{(2r)}) \in \mathbb{A}^{2nr}(\F_q) \setminus X_1 (\F_q),$ we would have
 \[ |S(\xbf^{(1)}, \ldots, \xbf^{(2r)}) | \leq Cq^{n/2},\]
 but for those $(\xbf^{(1)}, \ldots, \xbf^{(2r)}) \in X_1 (\F_q)$ we  could have an upper bound as bad as the trivial $(\# \F_q)^n$.
We would then have only the instance $j=1$ of Theorem \ref{thm_Xu}, namely
\[ \# \left\{ \{\xbf\}
\in (\zerobf,\kbf]^{2r} : \left|  \sum_{\mbf \modd{q}}
\chi(F_{\{\xbf\}}(\mbf))
\right|  > C q^{n/2} \right\}
 \leq C''  \| \kbf \|^{2r} k_1^{-\theta_1}.\]
 In place of (\ref{Y_original}) and Lemma \ref{moment_bound} we would now have 
 \begin{align*}
 \sum_{\mbf\modd{q}}\abs{S(F;\mbf,\kbf)}^{2r}
&\leq  \sum_{\bstack{\{\xbf\}\in (\zerobf,\kbf]^{2r}}{\{\xbf \} \in X_0 \setminus X_1}} \abs*{\sum_{\mbf\modd{q}}\chi(F_{\{\xbf\}}(\mbf))} 
+ \sum_{\bstack{\{\xbf\}\in (\zerobf,\kbf]^{2r}}{\{\xbf \} \in X_1}}\abs*{\sum_{\mbf\modd{q}}\chi(F_{\{\xbf\}}(\mbf))}  \\
& \ll \|\kbf\|^{2r} C q^{n/2} + C'' \| \kbf \|^{2r}q^n k_1^{-\theta_1}. 
\end{align*}
The second term has the worst growth $q^n$ appearing in (\ref{MR_weak1}) combined with the least savings $k_1^{-\theta_1}=B_{n,r,1}(\kbf)^{-1}$.
Proceeding with Menchov-Rademacher and the remaining argument we would obtain
\[ |S(F;\Nbf,\Hbf)|^{2r} \ll (\log q)^{2r(n+1)}  \norm{\Hbf}^{2r-1}P^{-1}q^{n/2} (1+q^{n/2}\widetilde{B}_{n,r,1}(\Hbf/P)^{-1}).\]
The last factor is $1+q^{n/2}(H_1/P)^{-\theta_1}$, which we balance by choosing $P=H_1 q^{-n/(2\theta_1)}$. This is a smaller choice than (\ref{P_choice_n}), hence provides smaller savings; ultimately this yields the bound
\beq\label{no_strat}
  |S(F;\Nbf,\Hbf)|  \ll  \norm{\Hbf}^{1-1/2r}H_1^{-1/2r}  q^{\frac{n\theta_1+n}{4r \theta_1}}  ( \log q)^{n+1}.
  \eeq
This is worse than our main theorem by a factor of $q^{\frac{n-1}{4r\theta_1}}$.

\subsection{Proof of Corollaries \ref{cor_threshold_same} and  \ref{cor_threshold_diff} }\label{sec_strength_proof}

Below, we prove Corollaries \ref{cor_threshold_same} and  \ref{cor_threshold_diff}  simultaneously; for the case of Corollary \ref{cor_threshold_same}, simply set $c_0 =1$ in each instance below.  We recall from Theorem \ref{thm_main_mult} that
\beq\label{compare1}
\abs{S(F;{\bf N,H})}
	\ll \norm{\Hbf}^{1-1/(2r)}H_{\min}^{-1/(2r)}q^{\frac{n\Theta+1}{4r\Theta}} (\log q)^{n+1} ,
\eeq
 where for every $r \geq1$ we have set 
 \[
 \Theta=\Theta_{n,r}= 
	\floor*{\frac{r-1}{n-1}}.
	\] 
First let us determine for a given $n \geq 2$ the threshold governing for which lengths $\Hbf$ the bound (\ref{compare1}) is nontrivial, that is $o(\|\Hbf\|)$, under the assumption that $\|\Hbf \|^{1/n} = q^{\be}$ and $H_{\min} \gg \| \Hbf\|^{c_0/n} = q^{c_0 \be}.$  In order to satisfy the hypotheses of Theorem \ref{thm_main_mult} for parameters $n,r$, we specify that $\be \leq {1/2 + 1/(4\Theta)}$. Then the bound (\ref{compare1})   is nontrivial as long as 
\[  \parens*{1 - \frac{1}{2r}} n\be - \frac{c_0}{2r} \be + \frac{n \Theta + 1}{4r\Theta} < n\be,\]
that is,
\beq\label{be_req}
 \be > \be_{n,r} := \frac{n \Theta + 1}{2\Theta(n+c_0)} =\frac{1}{2} - \frac{c_0\Theta-1}{2\Theta(n+c_0)}.
 \eeq
Given $n \geq 2$ and $0<c_0 \leq 1$, as long as we take $r$ sufficiently large that $\Theta=\Theta_{n,r}>1/c_0$, we have $\be_{n,r} < 1/2$.   (In particular if $c_0=1$, note that $\Theta_{n,r}>1$ when $(r-1)/(n-1) \geq 2$, or equivalently, $r \geq 2n-1$.) On the other hand, note that for fixed $n \geq 2$, for all $r \geq 1$ 
\[ \frac{1}{2} - \frac{c_0\Theta-1}{2\Theta(n+c_0)} >  \frac{1}{2}  - \frac{c_0}{2(n+c_0)} \geq \frac{1}{2}  - \frac{1}{2(n+1)},\]
for all $0< c_0 \leq 1$, and this a limitation on the range of $\be$ for which the bound is nontrivial.

We now compute that the bound (\ref{compare1}) is of the form $\norm{\Hbf}q^{-\delta}$
where
\beq\label{del_dfn_theta}
 \delta=\frac{n+c_0}{2r}\beta-\frac{n\Theta+1}{4r\Theta} .
 \eeq
We make the approximation that $\Theta = (r-1)/(n-1)$, which is an identity when $n=2$, and will not be far from the truth, when we later take $r$ very large. Then we compute that as a function of $r$, $\del$ can be represented as 
\[ f_{a,b,c}(r) = a/r - (br-c)/(r(r-1))\]
where 
\[ a=(n+c_0)\be/2, \qquad b = n/4, \qquad c=1/4,\]
 and thus $\del$ attains a maximum at
\beq\label{r_choice}
 r = (a-b)^{-1} \{ (a-c) \pm \sqrt{(a-c)^2  - (a-c)(a-b)}\}.
 \eeq
To have $r>0$ we must have $(a-b)>0$, that is
\beq\label{lower_bound_be}
  \be > \be_{n,c_0} := \frac{n}{2(n+c_0)} = \frac{1}{2} - \frac{c_0}{2(n+c_0)},
  \eeq
  agreeing with our previous observation.
Thus from now on we assume $\be = \be_{n,c_0} + \kappa$ for some small $\kappa$, and we will study how $\del = \del_n(\kappa)$ behaves as $\kappa \maps 0$. From (\ref{del_dfn_theta}) we see that 
\[ \del = \del_n(\kappa) \approx  \frac{n+c_0}{2r}\kappa-\frac{1}{4r\Theta} \]
and as $\kappa \maps 0$ we see in the choice of $r$ given by (\ref{r_choice}) that we will take $r$ to be the integer closest to 
\[ r\approx \frac{n-1}{\kappa(n+c_0)} .\]
Plugging this choice of $r$ into $\del_n(\kappa)$ and using the approximation $\Theta \approx r/(n-1)$ (which is valid as $\kappa \maps 0$ since then $r \maps \infty$), we see that 
\beq\label{del_comparison_final}
 \delta_n \approx \frac{(n+c_0)^2}{4(n-1)}\kappa^2 .
 \eeq
Alternatively, we can encapsulate the restriction (\ref{lower_bound_be})   by recording it as the restriction
 \[ \|\Hbf \| H_{\min} \gg \| \Hbf \|^{1 + c_0/n}  \gg (q^{n(1+c_0/n)})^{\frac{1}{2} - \frac{c_0}{2(n+c_0)}} = q^{\frac{n}{2} + \frac{c_0(1-c_0)}{2(n+c_0)}}.
 \]
 Thus we will obtain a nontrivial bound as long as $\| \Hbf \| H_{\min} \gg q^{n/2 + \kappa}$ for some small $\kappa$.

\subsection{Proof of technical lemmas}\label{sec_tech}

\begin{proof}[Proof of Lemma \ref{lemma_A_D}]

 This argument originates in \cite[\S 4]{HB13} and is similar but not identical to lemmas in \cite{HBP15} and \cite{Pie16}; for completeness we provide an argument.
 The first property in Lemma \ref{lemma_A_D} is a direct result of the definition of $\Acal(\mbf)$.
Since each $\Acal(\mbf)$ is a non-negative integer, 
 \[ \sum_{\mbf} \Acal(\mbf) \leq \sum_{\mbf} \Acal(\mbf)^2,\]
 and it suffices to prove the third property. We write
\begin{eqnarray}
 \sum_{\mbf} \Acal(\mbf)^2 
 	&=& \sum_\mbf \# \{p, p', \abf, \abf' : m_i \leq \frac{N_i-a_iq}{p}  < m_i +H_i/P, m_i \leq \frac{N_i-a_i'q}{p'}  < m_i +H_i/P \}  \nonumber \\
	& \ll & \left( \prod_{i=1}^n \frac{H_i}{P} \right) \# \{p, p', \abf, \abf' : 0 \leq | \frac{N_i-a_iq}{p}  - \frac{N_i-a_i'q}{p'} | \leq H_i/P \} \nonumber \\
	& \ll &  \| \Hbf \| P^{-n}  \sum_{p, p' \in \Pcal} \Mcal(p,p'), \label{HPM}
	\end{eqnarray}
 where 
 \[ \Mcal(p,p') =  \# \{\abf, \abf', 0 \leq a_i < p, 0  \leq a_i' < p' : 0 \leq | \frac{N_i-a_iq}{p}  - \frac{N_i-a_i'q}{p'} | \leq H_i/P \}.\]
 
 First consider $p=p'$. Then 
 \begin{eqnarray*}
  \Mcal(p,p) & \leq&\# \{\abf, \abf' : |(N_i - a_iq) - (N_i - a_i'q)| \leq p (H_i/P) \leq 2H_i, i=1,\ldots,n\} \\
  	& \leq &  \# \{\abf, \abf' : |a_i - a_i'| \leq 2H_i/q < 2, i=1,\ldots,n\}.
	\end{eqnarray*}
	Here we have used $H_i < q$.
This shows that once $\abf$ is chosen, there are at most $3^n$ choices for $\abf'$, so that $\Mcal(p,p) \ll P^n$ and hence $\sum_{p=p' \in \Pcal}\Mcal(p,p') \ll P^{n+1}$, which suffices for our desired bound for (\ref{HPM}).

Next, consider the case $p \neq p'$. For each $i=1,\ldots, n$ we choose (by Bertrand's postulate) a prime $l_i$ such that 
\[ \frac{q}{H_i} < l_i \leq \frac{2q}{H_i}.\]
(Here we use the assumption that $H_i < q$ for each $i$.)
For each $i$, let $M_i = \left[ \frac{N_i l_i}{q} \right]$ or $M_i = \left[ \frac{N_i l_i}{q} \right]+1$, so that $l_i \ndiv M_i$. Then $|N_il_i/q - M_i | \leq 1$ implies that $|N_i - qM_i/l_i| \leq q/l_i$, so that 
\begin{eqnarray*}
 \Mcal(p,p') &\ll& \# \{ \abf, \abf' : \left| \frac{qM_i/l_i - a_iq}{p} - \frac{qM_i/l_i  - a_i'q}{p'} \right| \leq \frac{H_i}{P} + \frac{q}{l_ip} + \frac{q}{l_ip'}, i=1,\ldots,n \} \\
 	& \ll & \# \{ \abf, \abf' : | (p'- p)M_i - (a_ip' - a_i'p)l_i| \leq 12 P, i=1,\ldots, n\}.
 \end{eqnarray*}
 Given $p,p'$ and an integer $\del$, for each fixed index $i$ there is at most one way to choose $a_i, a_i'$ with $0 \leq a_i < p$, $0 \leq a_i' < p'$ such that $a_ip' -a_i'p = \del$. 
 Thus 
\[  \sum_{p \neq p' \in \Pcal} \Mcal(p,p') \ll \#\{ p \neq p' \in \Pcal, \ubf=(u_1,\ldots, u_n), |u_i| \leq 12P: 
   M_i(p'-p) \con u_i \modd{l_i}, i=1,\ldots,n \}.
\]
 Now we use the fact that $l_i \ndiv M_i$. Thus for a fixed $i$, the condition $M_i(p'-p) \con u_i \modd{l_i}$ determines $p'-p$ uniquely modulo $l_i$, and hence uniquely in $\Z$, as long as $P < l_i$, which is guaranteed by the assumption $P \leq q/H_i$, that is $PH_i < q$. In particular, there is at most one value for the difference $p'-p$ that will satisfy all $n$ conditions. So we may choose $p$ freely and then $p'$ is determined. As a result, after counting up the possible choices for $\ubf$, we conclude that 
   \[ \sum_{p \neq p' \in \Pcal} \Mcal(p,p') \ll P^{n+1}.\]
Applying this in (\ref{HPM}), we conclude that
\[  \sum_{\mbf} \Acal(\mbf)^2 \ll \| \Hbf \| P.\]
\end{proof}

 \begin{proof}[Proof of Lemma \ref{lemma_rearrange}]
We may restrict our attention to permutations that map $\{1,\ldots,M\}$ to itself, or equivalently, we may suppose going forward that $M=n$. For indeed, any indices $i$ that occur in the sums such that $i \leq M$ but $\sig(i) >M$, clearly contribute no more to the left-hand side than to the right-hand side, since $t_i \leq t_{\sig(i)}$. 

 Now let $\sig$ be the permutation that minimizes 
\beq\label{rearrangement}
a_1t_{\sig(1)} + \cdots + a_n t_{\sig(n)};
\eeq
 if there is more than one such permutation, we choose $\sig$ to be the one with the greatest number of fixed points. We will show that $\sig$ is the identity. For suppose otherwise, and let $i$ be the smallest index such that $\sig(i) \neq i$. Then $\sig(i)>i$ and hence $t_{\sig(i)} \geq t_i$. Furthermore, denoting by $k$ the index such that $\sig(k)=i$, we must also have that $k>i$ and hence 
$a_k \leq a_i$. We then see that $(t_{\sig(i)} - t_i) ( a_i - a_k) \geq 0$, or equivalently, 
$ t_{\sig(i)} a_i + t_i a_k \geq t_i a_i + t_{\sig(i)} a_k.
$
Define a new permutation $\sig'$ by $\sig'(u)=u$ for $u=1,\ldots, i$, $\sig'(k)=\sig(i)$, and $\sig'(u) = \sig(u)$ for all the remaining $u \in \{i+1,\ldots, n\} \setminus \{k\}$. Then we see that 
$
  t_{\sig(i)} a_i + t_{\sig(k)} a_k \geq t_{\sig'(i)} a_i + t_{\sig'(k)} a_k$
  so that (\ref{rearrangement}) does not increase in value if we replace $\sig$ by $\sig'$, and $\sig'$ must also be a minimizer. Yet $\sig'$ has one more fixed point than $\sig$, a contradiction. We conclude that $\sig$ is the identity.
\end{proof}

\begin{proof}[Proof of Lemma \ref{lemma_bound_sum}]
By the definition of the $\theta_j$, the sum over $j=0,1,\ldots, n-1$ takes the form
\beq\label{sum_j}
1 + \sum_{j=1}^{n-1}q^{j/2} K_1^{-\theta_j}   =  1 + \sum_{j=1}^{n-2}q^{j/2} K_1^{-j\theta_1}  + q^{(n-1)/2}K_1^{-(r-1)}    \leq 1 +  \sum_{j=1}^{n-1}(q^{1/2} K_1^{-\theta_1}  )^j
\eeq
in which $\theta_1 = \lfloor (r-1)/(n-1) \rfloor$.   Here we have used the fact that for $j=n-1$, $\theta_{n-1} = r-1 \geq (n-1) \lfloor (r-1)/(n-1) \rfloor = (n-1)\theta_1$.
Now we see from the right-most expression that under the assumption (\ref{K_relation}), all terms $j \geq 2$ are dominated by $j=0,1$. On the other hand, we see from the middle expression that if (\ref{K_relation}) does not hold, then that expression is $\gg 1$, as claimed. 

It remains to examine the terms with $j=n$, which we divide into the even and odd cases. For $n \geq 2$ even, the $j=n$ term is 
\[q^{n/2}(K_1 \cdots K_{n/2})^{-(2r-1)}  \leq q^{n/2} K_1^{-(2r-1)n/2} \leq (q^{1/2}K_1^{\theta_1})^n, \]
in which we have used the ordering $K_1 \leq K_2 \leq \cdots \leq K_n$ and the fact that $(2r-1)(n/2) \geq n(r-1)/(n-1) \geq n\theta_1$ holds when $n \geq 2$. Thus the above expression is $\ll 1$ under the assumption (\ref{K_relation}). 

For $n \geq 3$ odd, the $j=n$ term is 
	\[q^{n/2}(K_1 \cdots K_{(n-1)/2})^{-(2r-1)}K_{(n+1)/2}^{-r} \leq q^{n/2} K_1^{-\{(2r-1)(n-1)/2 + r\}}\leq (q^{1/2}K_1^{\theta_1})^n,  \]
	upon verifying that $n \geq 2$ suffices to show that  $(2r-1)(n-1)/2 + r \geq n (r-1)/(n-1)  \geq n\theta_1$. Under the assumption (\ref{K_relation}) we see that the $j=n$ term is also $\ll 1$. This concludes the proof of Lemma \ref{lemma_bound_sum}.

\end{proof}

\section{Appendix}\label{sec_compare}
\subsection{Fourier methods for incomplete sums that are not short}\label{sec_Fourier}
Roughly speaking, the threshold $H_i \leq q^{1/2}$ appears as a natural constraint of the ranges for which our main results hold. This is not a deficit, for recall that on the other side of this threshold, different methods, which also rely on Weil bounds, become feasible. To bound $S(F;\Nbf,\Hbf)$ in cases where $H_i \gg q^{1/2}$, an advantageous strategy is to ``complete the sum,'' writing 
\[ S(F;\Nbf,\Hbf) = \sum_{\bstack{\abf = (a_1,\ldots, a_n) }{a_i \modd{q}}} \chi (F(\abf))  \sum_{\bstack{\xbf \in \Z^n}{x_i \in (N_i,N_i+H_i]}} \mathbf{1}_{\xbf \con \abf \modd{q}}  .\]
One then expands the sum over $\xbf$ using 
\[ \mathbf{1}_{\xbf \con \abf \modd{q}} = \frac{1}{q^n}\sum_{\kbf \modd{q}} e_q(\kbf \cdot (\xbf - \abf))\]
so that 
\[ S(F;\Nbf,\Hbf) = \frac{1}{q^n} \sum_{\kbf \modd{q}} \sum_{\abf \modd{q}} \chi (F(\abf))e_q(-\kbf \cdot   \abf) \Xi_1(a_1/q) \cdots \Xi_n(a_n/q),\]
in which $\Xi(\al) = \min \{H_i, \|\al\|^{-1}\},$ where $\|\al\|$ denotes the distance from $\al$ to the nearest integer.
One then aims to show that under appropriate assumptions on the smoothness of $F \in \F_q[x_1,\ldots, x_n]$, for generic $\kbf$ a Weil bound applies so that the internal sum over $\abf$ is $O(q^{n/2})$. (Note that this does require the deep input of a Weil-strength bound for a multi-dimensional mixed character sum; see e.g. \cite{Kat06} for one such reference.) The resulting sum over $\kbf$ is  then expected to be roughly on the order of $O(q^{-n/2}\|\Hbf \| + q^{n/2} (\log q)^n)$, which is $o(\|\Hbf\|)$ in the case that $H_i \gg q^{1/2}$ for each $i=1,\ldots,n$,  that is, the setting that is complementary to that of this paper. In a hybrid case, in which some $H_i$ are smaller than $q^{1/2}$ and some are larger, one could adopt a hybrid strategy; this regime is closely related to \cite{Pie06,PieHB12a}.

\subsection{Conditional results:  assuming a stronger stratification result} \label{sec_improvements}

In our Theorem \ref{thm_main_mult}, the larger $\Theta=\Theta_{n,r}$ is as a function of $r$, the better the bound is asymptotically in $n$.
We briefly explore how one could hope to increase the value of $\Theta_{n, r}$. The key is to improve Theorem \ref{thm_Xu_original}, and hence Theorem \ref{thm_Xu}, by obtaining larger values for the codimensions $\theta_j$.

At present, Theorem \ref{thm_Xu} holds with $\theta_0=0$ and $\theta_n = nr$; for the intermediate values $1 \leq j \leq n-1$, we currently obtain values 
\beq\label{compare_theta_j}
\theta_j  = j \left\lfloor \frac{r-1}{n-1}  \right\rfloor \approx j\left(\frac{r-1}{n-1}\right).
\eeq
 However,
suppose that in the stratification result of Theorem \ref{thm_Xu_original} (and hence in Theorem \ref{thm_F_power} and its corollaries under the modified hypotheses), we were able to take the larger values
\beq\label{theta_conjec}
\theta_j^\sharp=jr, \qquad 1 \leq j \leq n.
\eeq
This is a natural hypothesis since it is the linear interpolation between $\theta_0=0$ and $\theta_n=nr$.  In fact, note from the definition (\ref{compare_theta_j}) that we very nearly achieve (\ref{theta_conjec}) in the case of $n=2$.

Supposing that we can take $\theta_j^\sharp$ as large as in (\ref{theta_conjec}) in Theorem \ref{thm_Xu_original}, we deduce that Theorem \ref{thm_Xu} would hold with the function $B_{n, r, j}(\kbf)$ replaced by the modified function $B^\sharp_{n, r, j}(\kbf)$ defined for $0 \leq j \leq n$ and $\kbf = (k_1,\ldots, k_n)$ with $k_1 \leq k_2 \leq \cdots \leq k_n$ by 
\[ B^\sharp_{n, r,j}(\kbf)  = \begin{cases}
	1 & \text{if $j=0$}\\
	(k_1 \cdots k_{j/2})^{2r} & \text{if $j \geq 1$ is even}\\
	(k_1 \cdots k_{(j-1)/2})^{2r} k_{(j+1)/2}^r& \text{if $j \geq 1$ is odd}.
\end{cases}
\]
Proceeding through the Burgess argument in this paper with the function $B_{n, r, j}(\kbf)$ replaced in each instance by $B^\sharp_{n, r, j}(\kbf)$, we would arrive at the analogue of (\ref{concluding}), which now takes the form
\beq\label{S_conditional}
  |S(F; \Nbf,\Hbf)|^{2r} \ll_{n, r, \Del, D} (\log q)^{2r(n+1)}  \norm{\Hbf}^{2r-1}P^{-1}q^{n/2} \sum_{j=0}^n q^{j/2}\widetilde{B}^\sharp_{n,r,j}(\Hbf/P)^{-1},\eeq
in which for any $\kbf$ with $k_1 \leq k_2 \leq \cdots \leq k_n$ we define
\[ \widetilde{B}^\sharp_{n, r,j}(\kbf)  = \begin{cases}
	1 & \text{if $j=0$}\\
	(k_1 \cdots k_{j/2})^{2r-1} & \text{if $j \geq 1$ is even}\\
	(k_1 \cdots k_{(j-1)/2})^{2r-1} k_{(j+1)/2}^r& \text{if $j \geq 1$ is odd}.
\end{cases}
\]
Recall that we assume $H_1 \leq H_2 \leq \cdots \leq H_n$.
We choose (cf. (\ref{P_choice_1}) and (\ref{P_choice_n})) $P$ to be an integer such that 
\[ \frac{H_1}{2q^{1/(2r-1)}} \leq P \leq \frac{H_1}{q^{1/(2r-1)}} \]
which balances the $j=0$ and $j=2$ contributions. Under this choice, a simple computation shows that for each $j=0, \ldots, n$ we verify that
\[ q^{j/2} \widetilde{B}^\sharp_{n,r,j}(\Hbf/P)^{-1} \leq 1,\]
upon using the fact that $H_1 \leq H_2 \leq \ldots \leq H_n$ and the definition of $ \widetilde{B}^\sharp_{n, r,j}$ above.
Applying this in (\ref{S_conditional}) would give
\[ |S(F; \Nbf,\Hbf)| \ll_{n, r, \Del, D}   \norm{\Hbf}^{1-\frac{1}{2r}}H_1^{-\frac{1}{2r}} q^{\frac{n \Theta^\sharp + 1}{4r\Theta^\sharp}} (\log q)^{n+1}
\]
with 
\[ \Theta^\sharp=(2r-1)/2 = r-1/2.\]

We can compute that this is nontrivial for $\Hbf$  satisfying the analogues of (\ref{cor_range_beta_n}) or (\ref{cor_range_beta_n'}) with $\Theta^\sharp$ in place of $\Theta$; in the limit as $n \maps \infty$ 
we quantify the strength of the bound near the threshold $\beta_n=1/2 - 1 /(2(n+1))$. Letting $\be = \be_n + \kappa$, then our bound is of the form $\norm{\Hbf}q^{-\delta}$, where
\[ \delta=\frac{n+1}{2r}\kappa-\frac{1}{4r\Theta^\sharp} \approx
\frac{n+1}{2r}\kappa-\frac{1}{4r^2}.\]
The maximum
\[\delta\approx \frac{(n+1)^2}{4}\kappa^2\]
 is achieved when
\[ r \approx \frac{1}{\kappa(n+1)}. \]
Thus this conjectural improvement to the stratification would yield a stronger savings near the threshold $\be_n$, but would not alter the fundamental threshold $\be_n$. Similar computations to those above show that even if we could conjecturally improve the $\theta_j$ values to the strongest possible values $\theta_j^\flat=nr$ for $1\le j\le n$, this bound will not be improved substantially, and the threshold $\be_n$ will not change.

\section*{Acknowledgements}
The authors thank D. R. Heath-Brown and M. Larsen for helpful advice on this project.
Pierce is partially supported by NSF CAREER grant DMS-1652173,  a Sloan Research Fellowship and the AMS Joan and Joseph Birman Fellowship. 
Xu has been partially supported by NSF DMS-1702152.
 
\bibliographystyle{alpha}
\bibliography{NoThBibliography}

\end{document}